\documentclass[12pt,reqno]{amsart}
\usepackage{amsmath,amssymb,verbatim,comment,color}
\usepackage[pagebackref,pdftex,hyperindex]{hyperref}

\usepackage[margin=3cm]{geometry}
\usepackage{color}

\newcommand{\C}{\mathbb{C}}

\renewcommand{\P}{\mathbb{P}}
\newcommand{\Q}{\mathbb{Q}}
\newcommand{\R}{\mathbb{R}}
\renewcommand{\S}{\mathbb{S}}

\newcommand{\bp}{\bar{\partial}}

\newcommand{\fg}{\mathfrak{g}}
\newcommand{\fh}{\mathfrak{h}}

\newcommand{\ft}{\mathfrak{t}}
\newcommand{\fs}{\mathfrak{s}}

\newcommand{\cB}{\mathcal{B}}

\newcommand{\cE}{\mathcal{E}}

\newcommand{\cH}{\mathcal{H}}

\newcommand{\cL}{\mathcal{L}}

\newcommand{\cX}{\mathcal{X}}

\renewcommand{\d}{\delta}

\newcommand{\s}{\sigma}

\newcommand{\la}{\lambda}

\renewcommand{\phi}{\varphi}

\renewcommand{\leq}{\leqslant}
\renewcommand{\geq}{\geqslant}

\newcommand{\norm}[1]{\left\|#1\right\|}

\newcommand{\NA}{\mathrm{NA}}
\renewcommand{\Re}{\mathrm{Re}}

\DeclareMathOperator{\Aut}{Aut}

\DeclareMathOperator{\dd}{\sqrt{-1}\partial \bar{\partial}}

\DeclareMathOperator{\Hom}{Hom}

\DeclareMathOperator{\lct}{lct}

\DeclareMathOperator{\Ric}{Ric}

\DeclareMathOperator{\supp}{supp}

\DeclareMathOperator{\Tr}{Tr}

\numberwithin{equation}{section}       

\newtheorem{prop} {Proposition} [section]
\newtheorem{thm}[prop] {Theorem} 
\newtheorem{dfn}[prop] {Definition}
\newtheorem{lem}[prop] {Lemma}
\newtheorem{cor}[prop]{Corollary}

\newtheorem{exam}[prop]{Example}
\theoremstyle{remark}
\newtheorem*{ackn}{\bf{Acknowledgment}} 
\newtheorem{rk}[prop]{Remark}

\newcommand{\del}{\partial}

\newcommand{\dbar}{\overline{\del}}
\newcommand{\ddt}{\frac{d}{dt}}

\renewcommand{\leq}{\leqslant}
\renewcommand{\geq}{\geqslant}

\renewcommand{\epsilon}{\varepsilon}
\renewcommand{\phi}{\varphi}

\title[The inverse Monge-Amp\`ere flow]{The inverse Monge-Amp\`ere flow and applications to K\"ahler-Einstein metrics} 
\date{\today} 

\author[T. C. Collins]{Tristan C. Collins}
\address{Department of Mathematics\\
 Harvard University\\
 One Oxford St.\\
 Cambridge\\
 MA02138\\
 USA}
\email{tcollins@math.harvard.edu}

\author[T. Hisamoto]{Tomoyuki Hisamoto}
\address{Graduate School of Mathematics\\
  Nagoya University\\
  Furocho\\
  Chikusa\\
  Nagoya\\ 
  Japan}
\email{hisamoto@math.nagoya-u.ac.jp}

\author[R. Takahashi]{Ryosuke Takahashi}
\address{Mathematical Institute\\
 Tohoku University\\
 Aoba\\
 Aramaki\\
 Aoba-ku\\
 Sendai\\
 Japan}
\email{ryosuke.takahashi.a7@tohoku.ac.jp}

\begin{document}

\maketitle

\begin{abstract}
We introduce the inverse Monge-Amp\`ere flow as the gradient flow of the Ding energy functional on the space of K\"ahler metrics in $2 \pi \lambda c_1(X)$ for $\lambda=\pm 1$.  We prove the long-time existence of the flow.  In the canonically polarized case, we show that the flow converges smoothly to the unique K\"ahler-Einstein metric with negative Ricci curvature.  In the Fano case, assuming the $X$ admits a K\"ahler-Einstein metric, we prove the weak convergence of the flow to the K\"ahler-Einstein metric.  In general, we expect that the limit of the flow is related with the optimally destabilizing test configuration for the $L^2$-normalized non-Archimedean Ding functional.  We confirm this expectation in the case of toric Fano manifolds.
\end{abstract}

\section{Introduction}

Let $(X,L)$ be a polarized K\"ahler manifold with $-K_{X}= \lambda L$ with $\lambda \in \{-1,0,+1\}$.  Beginning with the work of Calabi it has been a fundamental problem in K\"ahler geometry to prove the existence of a K\"ahler metric $\omega \in 2 \pi c_{1}(L)$ satisfying Einstein's equation
\[
\Ric(\omega) = \lambda \omega.
\]
When $\lambda =0,-1$, this problem was famously resolved by Yau \cite{Y} who showed that there is {\em always} a K\"ahler-Einstein metric in the class $2 \pi c_1(L)$ (see also Aubin \cite{A} when $\lambda=-1$).  In contrast, when $\lambda =+1$-- the Fano case-- it has been known since work of Matsushima \cite{Mat57} that K\"ahler-Einstein metrics need not exist; for example, the blow-up of $\mathbb{P}^2$ in a point does not admit any K\"ahler-Einstein metric.  The Yau-Tian-Donaldson conjecture predicts that the existence of a K\"ahler-Einstein metric in the class $2 \pi c_1(X)$ is equivalent to the K-stability of $(X,-K_X)$.  This conjecture was resolved in fundamental work of Chen-Donaldson-Sun \cite{CDS15a, CDS15b, CDS15c}. 

\begin{thm}[Chen-Donaldson-Sun]
There exists a K\"ahler-Einstein metric in the class $2 \pi c_1(X)$ if and only if $(X,-K_{X})$ is $K$-polystable.
\end{thm}

An important open question in this direction is to understand what happens when $(X,-K_{X})$ is not $K$-stable.  In analogy with the Harder-Narasimhan filtration of an unstable vector bundle, or with optimally destabilizing one-parameter subgroups in GIT, one expects an optimal degeneration of unstable varieties.  This problem has recently attracted a great deal of attention.  One approach to this problem has been to use the K\"ahler-Ricci flow.  Building on work of Chen-Wang \cite{CW14},  Chen-Sun-Wang \cite{CSW15} showed that the K\"ahler-Ricci flow produces a degeneration of $X$.  Dervan-Sz\'ekelyhidi \cite{DS17} showed that this degeneration could be interpreted as ``maximally destabilizing" in the sense that it maximizes the so-called $H$-invariant, an algebro-geometric invariant associated to the asymptotics of the $H$-functional.  As far as we can tell, the $H$-functional appeared originally in the work of Ding-Tian \cite{DT92b}.  The functional was shown to be monotonic along the K\"ahler-Ricci flow in \cite{P, PSSW09}, and was studied as a functional on the space of K\"ahler metrics by He \cite{He}.  As pointed out by Dervan-Sz\'ekelyhidi \cite[Lemma 2.5]{DS17}, the $H$-invariant only detects $K$-semistability.  In this paper we introduce a new approach to this problem based on a natural parabolic flow.

We first provide a conceptual motivation for our approach.  Work of Phong-Ross-Sturm \cite{PRS08}, Paul-Tian \cite{PT09}, Berman \cite{Berm16}, and more recently Berman-Boucksom-Jonsson \cite{BBJ15} has shown that there is a deep connection between K\"ahler-Einstein metrics, $K$-stability, and a certain infinite dimensional variational framework.  Fix once and for all a K\"ahler metric $\omega_0 \in 2 \pi c_1(L)$.  Then any K\"ahler metric in $2 \pi c_1(L)$ can be written as $\omega_\phi=\omega_0 + \dd \phi $ for some $\phi \in C^\infty(X; \R)$ which is unique up to a constant.  It is therefore natural to introduce the space
\begin{equation}\label{eq: spaceOfMetrics}
\cH_{\omega_0} := \left\{ \phi \in C^\infty(X; \R) : \omega_0 +\dd \phi >0 \right\}
\end{equation}
which we regard as the space of K\"ahler metrics in $2 \pi c_1(L)$; to ease notation we will often drop the subscript $\omega_0$. For $\phi \in \cH$ define the normalized Ricci potential $\rho_{\phi}$ by 
\begin{equation*}
\Ric{\omega_\phi} -\la \omega_\phi = \dd \rho_\phi, \qquad \int_X e^{\rho_\phi} \omega_\phi^n =V. 
\end{equation*}
Here the volume $V:= \int_X \omega_\phi^n$ of $(X, L)$ is a topological constant. By a standard calculation we have
\begin{equation}\label{eq: Ricci pot}
\rho_{\phi} = -\lambda \phi - \log \left( \frac{\omega_{\phi}^{n}}{\omega_0^n}\right) +\rho_0 +c
\end{equation}
for a constant $c$.  Unless otherwise noted, we will omit the subscript and denote the metric by $\omega=\omega_\phi$ and its Ricci potential by $\rho$ for simplicity. Define an energy functional $D:\cH \rightarrow \R$ by its variation according to
\begin{equation}\label{Ding variation}
\d D(\phi) = \frac{1}{V}\int_X (\d \phi) (e^\rho-1) \omega_\phi^n.   
\end{equation}
This natural energy functional was known to experts in the field in the '70s \cite{Y1}, but seems to have be written down first in the paper of Bando-Mabuchi \cite{BM85}, \cite{BM86} and subsequently by Ding \cite{D88}. It was studied by Ding-Tian \cite{DT92}, and is now referred to as the Ding functional.  The functional $D$ is a natural energy functional which has K\"ahler-Einstein metrics as its critical points.  Explicitly, fixing the reference metric $\omega_0$ and integrating, we can write the Ding functional as
\[
D(\phi) =- \frac{1}{\lambda} \log \left( \int_{X} e^{-\lambda \phi +\rho_0} \omega_0^{n}\right)  -E(\phi) 
\]
when $\lambda=-1,+1$, and
\[
D(\phi) =  \frac{1}{V}\int_{X} \phi e^{\rho_0}\omega_0^n -E(\phi) 
\]
when $\lambda =0$.  Here $E(\phi)$ is the Aubin-Yau functional defined by the variational formula
\[
\delta E(\phi) =  \frac{1}{V}\int_{X} \delta \phi \omega_{\phi}^{n}.
\]

Following the work of Mabuchi \cite{Mab87}, Semmes \cite{Sem92} and Donaldson \cite{Don99}, we endow the space $\cH$ with the structure of an infinite dimensional Riemannian manifold. Define an inner product on $T_{\phi} \cH \cong C^{\infty}(X; \R)$ by
\[
\langle f ,g \rangle_{\phi} = \int_{X} fg \omega_{\phi}^{n}.
\]
This Riemannian structure leads to a notion of geodesics, and the space $\cH$ can be viewed as an infinite dimensional symmetric space.  Let us restrict our attention to the Fano case, when $\lambda =1$.  A fundamental result of Berndtsson \cite{Bern11} says that the Ding functional is convex along geodesics.  Donaldson \cite{Don15} showed that the Ding functional is the Kempf-Ness functional for a certain infinite dimensional GIT problem, and in particular, K\"ahler-Einstein metrics are zeroes of the associated moment map.  In this picture, geodesics in the space of K\"ahler metrics are the natural analog of one-parameter subgroups, and the standard GIT framework suggests that we should study the slope of the Ding functional at infinity along geodesic rays in the space $\cH$.  Furthermore, the gradient flow of the Ding functional provides a natural approach to constructing optimal destabilizers.  

A similar GIT framework exists in the setting of holomorphic vector bundles.  In this case the Kempf-Ness functional is given by the Donaldson functional on the space of hermitian metrics.  The gradient flow of the Donaldson functional is the Donaldson heat flow which is known to converge to the Harder-Narasimhan filtration \cite{Don85, UY, DW, J15, J16}.

Going back to work of Phong-Sturm \cite{PS06, PS07}, it has been understood that one way to produce geodesic rays in $\cH$ is to use {\em test configurations} $(\cX,\cL)$ (see section~\ref{sec: stable}). This provides a link between the formal GIT picture and algebro-geometric stability.  Building on work of Phong-Ross-Sturm \cite{PRS08}, and Paul-Tian \cite{PT09},  Berman \cite{Berm16} described the asymptotic slope of the Ding functional along a geodesic ray associated to a test configuration.  Berman showed that when the singularities of the test configuration are sufficiently mild, the slope of the Ding functional is in fact equal to the Donaldson-Futaki invariant, and hence the Ding functional detects $K$-stability.  Boucksom-Hisamoto-Jonsson \cite{BHJ15} showed that the asymptotic slope of the Ding functional along a geodesic associated to a test-configuration is equivalent to to the algebraic invariant $D^\NA(\cX, \cL)$, the non-Archimedean counterpart of the Ding functional.  
 
Motivated by this infinite dimensional GIT picture, we introduce the Ricci-Calabi energy
\begin{equation}\label{Ricci Calabi}
R(\phi) =  \frac{1}{V}\int_X (e^\rho-1)^2 \omega_\phi^n.
\end{equation}
By a simple application of H\"older's inequality, the second author proved the lower bound estimate for the Ricci-Calabi energy \cite{His12}; 
\begin{equation}\label{lower bound of Ricci-Calabi}
\inf_\phi R(\phi)^{\frac{1}{2}} \geq \sup_{(\cX, \cL)}\frac{-D^\NA(\cX, \cL)}{\norm{(\cX, \cL)}}, 
\end{equation} 
where the right-hand side is the normalized non-Archimedean Ding invariant of $(\cX,\cL)$.  This can be viewed as an infinite dimensional generalization of the moment-weight inequality of GIT.  For holomorphic vector bundles on a curve the moment-weight inequality is the Atiyah-Bott formula \cite{AB}, and this was extended to higher dimensions by Jacob \cite{J16}.  Donaldson \cite{Don05} proved a similar inequality in the setting of constant scalar curvature K\"ahler metrics which has played a fundamental role in the development of the Yau-Tian-Donaldson conjecture.

In this paper we study the gradient flow of the Ding functional with respect to the Riemannian structure on $\mathcal{H}$.  That is, we study the flow
\begin{equation}\label{Mabuchi flow}
\ddt \phi = 1-e^{\rho}.
\end{equation}
We refer to this flow as the inverse Monge-Amp\`ere flow, or $MA^{-1}$-flow for short, but note that Mabuchi originally studied self-similar solutions of the flow \cite{Mab01}.  Note that the normalized K\"ahler-Ricci flow can be written as
\[
\ddt \phi = -\rho
\]
and hence, at a formal level, the inverse Monge-Amp\`ere flow is similar to the K\"ahler-Ricci flow whenever the Ricci potential is small.  On the other hand, when the Ricci potential is not small (as will be the case when $X$ does not admit a K\"ahler-Einstein metric), the flows are rather different.

Writing the $MA^{-1}$-flow in terms of the potential leads to the parabolic PDE
\begin{equation}\label{eq: DMflow}
\ddt \phi = 1- \frac{\omega_0^n}{\omega_{\phi}^n} e^{-\lambda\phi + \rho_0 +c(t)},
\end{equation}
for a time dependent constant $c(t)$.  From a PDE point of view, the $MA^{-1}$-flow, and the K\"ahler-Ricci flow can be viewed as different choices of parabolic complex Monge-Amp\`ere equation, with the $MA^{-1}$-flow being more in line with the parabolic Monge-Amp\`ere equation proposed by Krylov \cite{Kry76}.  We remark that inverse Monge-Amp\`ere flows have recently appeared in several contexts in complex geometry \cite{PPZ, PPZa}.  The flow~\eqref{eq: DMflow} is clearly parabolic, and hence the short-time existence follows automatically from general theory.  When $\la=0$ the long-time existence and convergence of the flow was obtained by Cao-Keller \cite{CK13}, with similar results in \cite{FLM11, FL12}.  We shall focus only on the case when $\la= \pm 1$.  When $\la =-1$ we prove

\begin{thm}\label{thm: mainCanonPolar}
Let $X$ be a canonically polarized K\"ahler manifold.  For any initial metric $\omega_{\phi_0}$ the inverse Monge-Amp\`ere flow exists for all time and converges in the $C^{\infty}$ topology to the unique K\"ahler-Einstein metric in $-2 \pi c_1(X)$.
\end{thm}

This gives another proof of Yau \cite{Y} and Aubin's \cite{A} existence theorem for K\"ahler-Einstein metrics with negative Ricci curvature.  In the Fano setting we prove

\begin{thm}\label{thm: mainFano}
Let $X$ be a Fano manifold.  For any initial metric $\omega_{\phi_0}$ the inverse Monge-Amp\`ere flow starting from $\phi_0$ exists for all time.  Furthermore, if $X$ has no holomorphic vector fields, and admits a K\"ahler-Einstein metric, then $\phi(t)$ converges to the K\"ahler-Einstein potential $\phi_{KE}$ in $L^{p}$ for all $p$, and in the strong topology.  If $X$ is K\"ahler-Einstein with holomorphic vector fields then the flow converges weakly in the sense of currents, modulo the action of $\Aut_{0}(X)$, to a K\"ahler-Einstein metric.
\end{thm}

We refer the reader to section~\ref{sec: DMflowFano} for the definition of the strong topology.  In contrast with the setting of the K\"ahler-Ricci flow, the long-time existence of the $MA^{-1}$-flow requires some non-trivial work as ODE techniques and the maximum principle are insufficient to rule out finite time singularities.  We note that the proof of the convergence of the flow to the K\"ahler-Einstein metric depends only on the coercivity (or modified coercivity when $\Aut_0(X) \ne 0$) of the Ding functional.

Next, we study the behavior of the $MA^{-1}$-flow on toric Fano manifolds in the spirit of Sz\'ekelyhidi's work on the Calabi flow \cite{Sze08}.  Let $(X,-K_{X})$ be a toric Fano manifold, and let $P$ denote the associated moment polytope.  In this setting each test configuration is identified with a unique rational piecewise-linear convex function $f : P \rightarrow \R$ \cite{Don02}.  It was shown by Yao \cite{Yao17} that the non-Archimedean Ding functional, and norm of the test configuration corresponding to $f$ are given by
\begin{equation*}
D^\NA(f)= -f(0) + \frac{1}{V_P} \int_Pf, \ \ \  \norm{f}^2= \int_Pf^2. 
\end{equation*} 
Let $e: P \rightarrow \mathbb{R}$ be the affine function associated to the extremal vector field (in the sense of Calabi) in $\mathfrak{t} = {\rm Lie}(\mathbb{C}^{*})^n$ and let $\ell:=e+1$. If the modified Ding invariant 
\begin{equation*}
D^\NA_{\ell}(f):=-f(0)  +\frac{1}{V_P} \int_P f  \ell
\end{equation*} 
is semipositive for any rational piecewise-linear convex function $f$, we say that $P$ is relatively Ding-semistable.  Yao \cite{Yao17} showed that if $(X,-K_{X})$ is not relatively Ding-semistable, then, up to scaling, there exists a unique convex function $d$, which is the maximum of two affine functions, achieving equality in~\eqref{lower bound of Ricci-Calabi}. We call $d$ the optimal destabilizer. We remark that $d$ may not be rational \cite{Yao17}, but does give rise to a degeneration.

\begin{thm}\label{thm: toric thm}
Let $X$ be a toric Fano manifold with associated moment polytope $P$.  For any $(S^1)^n$ invariant K\"ahler metric $\omega$, we write $\sigma$ for the function $e^{\rho}$ written on $P$ via the moment map associate to $\omega$.
\begin{enumerate}
\item If $P$ is relatively Ding-semistable, then $\s-1$ converges along the $MA^{-1}$-flow to the extremal affine function $e$ in $L^2(P)$. 
\item If $P$ is relatively Ding-unstable, then denoting by $d$ the optimal destabilizer for the $L^2$-normalized Ding invariant, $\s-1$ converges along the $MA^{-1}$-flow to $d+e$ in the Hilbert space $L^{2}(P)$.
\end{enumerate}
In particular, along the $MA^{-1}$-flow starting from any $(S^1)^n$ invariant metric we have
\[
\lim_{t\rightarrow \infty} R(\phi(t))^{1/2} =  \sup_{(\cX, \cL)}\frac{-D^\NA(\cX, \cL)}{\norm{(\cX, \cL)}}.
\]
\end{thm}

This theorem shows that, in the toric case, the $MA^{-1}$-flow provides a natural deformation of $X$ to a maximally destabilizing degeneration.  We note that in general the optimal destabilizer produced by the $MA^{-1}$-flow is different from the optimal destabilizer produced by the K\"ahler-Ricci flow.

The organization of this paper is as follows.  In section~\ref{sec: DMflow} we discuss some generalities, and collect some useful evolution equations and estimates along the inverse Monge-Amp\`ere flow.  In section~\ref{sec: DMflowCan} we prove the long-time existence and convergence of the flow on canonically polarized K\"ahler manifolds.  In section~\ref{sec: DMflowFano} we prove the long-time existence of the flow on Fano manifolds.  We also prove the weak convergence of the flow on K\"ahler-Einstein Fano manifold.  In section~\ref{sec: stable} we discuss some generalities concerning $K$-stability, Ding stability, relative Ding stability, and Mabuchi solitons.  Finally, in section~\ref{sec: toric} we analyze the flow on toric manifolds and prove Theorem~\ref{thm: toric thm}.  We conclude with some remarks concerning the general picture about the behavior the flow suggested by the toric case.

\begin{ackn}
The authors express their gratitude to Shigetoshi Bando, Robert Berman, S\'ebastien Boucksom, Ryoichi Kobayashi and Yuji Odaka for helpful conversations. T.C.C and T. H. would like to thank Chalmers University for their hospitality during a visit in May, 2017, where this project started.  T.C.C was supported in part by National Science Foundation grant DMS-1506652, the European Research Council and the Knut and Alice Wallenberg Foundation. T.H was supported by JSPS KAKENHI Grant Number 15H06262 and 17K14185. R.T was supported by Grant-in-Aid for JSPS Fellows Number 16J01211.
\end{ackn}

\section{The inverse Monge-Amp\`ere flow}\label{sec: DMflow}

Fix a polarized K\"ahler manifold $(X,L)$ with $\lambda L= -K_{X}$, $\lambda =\pm 1$.  As before we fix a K\"ahler metric $\omega_0 \in 2 \pi c_1(L)$, and consider the space of K\"ahler metrics $\cH$ defined in \eqref{eq: spaceOfMetrics}.   Recalling the definition of the Ricci potential, the $MA^{-1}$-flow can be written as
\[
\ddt \phi = 1-e^{\rho}
\]
where $\rho$ is the Ricci potential of $\phi$, normalized by $\int_{X}e^{\rho} \omega_{\phi}^n = \int_{X}\omega_0^n = V$.  Writing this out in terms of the potential $\phi$ we get
\begin{equation}\label{eq: DMflow2}
\ddt \phi = 1-\frac{\omega_0^n}{\omega_{\phi}^n}e^{-\lambda \phi + \rho_0 +c(t)}, \qquad c(t) = -\log \left(\frac{1}{V} \int_{X}e^{-\la \phi +\rho_0} \omega_0^n \right).
\end{equation}
The Ding functional and the Ricci-Calabi energy are respectively given by
\[
\begin{aligned}
D(\phi) &= -E(\phi) - \frac{1}{\lambda} \log \left( \int_{X} e^{-\lambda \phi +\rho_0} \omega_0^{n} \right)\\
& = -E(\phi) +\frac{1}{\lambda}c(t) -\frac{1}{\lambda}\log V,\\
R(\phi) &= \frac{1}{V} \int_{X}(1-e^{\rho})^{2} \omega_{\phi}^{n}.
\end{aligned}
\]
Along the flow we have
\[
\ddt E(\phi) = \frac{1}{V}\int_{X} \dot{\phi} \omega_{\phi}^{n} = 1 -\frac{1}{V} \int_{X} e^{\rho}\omega_{\phi}^n = 0
\]
and 
\[
\ddt D(\phi) = -\frac{1}{V} \int_{X} (1-e^{\rho})^{2}\omega_{\phi}^{n} = -R(\phi).
\]
Combining these basic observations we have
\begin{lem}\label{lem: BasicProps}
Along the inverse Monge-Amp\`ere flow we have
\begin{enumerate}
\item[(i)] The Aubin-Yau energy $E(\phi)$ is constant.
\item[(ii)] The Ding functional is monotonically decreasing.
\item[(iii)] The constant $c(t)$ defined in~\eqref{eq: DMflow2} satisfies
\[
c(t) = \lambda D(\varphi) + \lambda E(\varphi_0) + \log V.
\]
In particular, $c(t)$ is monotonically increasing when $\lambda =-1$, and monotonically decreasing when $\lambda =+1$.
\end{enumerate}
\end{lem}

Next, we compute the variation of the Ricci-Calabi energy $R(\phi)$.  We begin by computing the variation of the normalized Ricci potential.

\begin{lem}
The variation of the normalized Ricci potential $\rho = \rho_{\phi}$ is 
\[
\delta \rho = -\Delta_{\omega_{\phi}} \delta \phi - \lambda \left( \delta \phi - \frac{1}{V} \int_{X} \delta{\phi} e^{\rho} \omega_{\phi}^{n} \right).
\]
\end{lem}
\begin{proof}
From the definition of $\rho$ we have
\[
\delta \rho = -\Delta_{\omega_\phi} \delta \phi - \lambda \delta \phi +a
\]
for some constant $a$.  On the other hand, since $\int_{X}e^{\rho} \omega_{\phi}^n = V$ we have
\[
\int_{X} (\delta \rho + \Delta_{\omega_{\phi}} \delta \phi) e^{\rho} \omega_{\phi}^n =0
\]
and so $a$ is determined by integration and the lemma is proved.
\end{proof}

With this computation we compute the variation of the Ricci-Calabi energy.

\begin{prop}\label{prop: RCvar}
The Ricci-Calabi energy satisfies
\[
\delta R(\phi) = \frac{-2}{V} \int_{X} \delta \phi \left( \Delta_{\omega_{\phi}}e^{\rho} + (\dbar e^{\rho} ,\dbar \rho)_{\omega_{\phi}} + \lambda (e^{\rho} - \frac{1}{V} \int_{X}e^{2\rho} \omega_{\phi}^n) \right) e^{\rho}\omega_{\phi}^{n}.
\]
\end{prop}
\begin{proof}
With the above lemma it is straightforward to see 
\begin{align*}
\d R(\phi)
&= \frac{2}{V}\int_X (\d \rho) (e^\rho-1) e^\rho \omega^n 
+\frac{1}{V}\int_X (\Delta \d\phi) (e^\rho-1)^2 \omega^n \\
&= \frac{2}{V}\int_X \bigg( -\Delta (\d \phi) - \la \big(\d\phi - \frac{1}{V}\int_X (\d\phi)e^\rho \omega^n\big)\bigg) (e^\rho-1) e^\rho \omega^n 
+\frac{1}{V}\int_X (\Delta \d\phi) (e^\rho-1)^2 \omega^n \\
&= -\frac{1}{V}\int_X (\Delta \d\phi)(e^{2\rho}-1)\omega^n
-\frac{2}{V}\int_X \la\big(\d\phi - \frac{1}{V}\int_X (\d\phi)e^\rho\omega^n \big)(e^\rho-1)e^\rho \omega^n. 
\end{align*}
Integration by parts yields 
\begin{equation*}
\int_X (\Delta \d\phi)(e^{2\rho}-1)\omega^n 
=2\int_X (\d\phi)(\Delta e^\rho + (\bp e^\rho, \bp \rho)_\omega) e^\rho\omega^n,  
\end{equation*}
and we obtain the formula. 
\end{proof}
The variation of the Ricci-Calabi energy will play an important role in what is to follow, and so we spend a moment to expand on the above formula.  Introduce the twisted Laplacian on functions $f \in C^{\infty}(X,\mathbb{C})$ defined by
\[
L_{\rho}f = e^{-\rho} \nabla_{i} \left(g^{i\bar{j}}e^{\rho} \nabla_{\bar{j}} f\right) = \Delta_{g} f + g^{i\bar{j}}\del_{\bar{j}} f \del_i\rho = \Delta_g f + (\dbar \rho, \dbar f)_{\omega}.
\]
This operator has played an important role in the study of the K\"ahler-Ricci flow (see e.g. \cite{PS06a, PSSW09, Zh11}), and it plays a similarily important role in determining the behavior of the Ding functional along the $MA^{-1}$-flow.  Let us introduce the measure $d\mu = e^{\rho} \omega^n$, and consider the Hilbert space $L^{2}(X,d\mu)$.  Then the variation of the Ricci-Calabi energy can be written succinctly as
\[
\delta R(\phi) = -\frac{2}{V} \int_{X} \delta \phi \left( L_{\rho}\tilde{f} + \lambda \tilde{f} \right) d\mu
\]
if we take
\[
\tilde{f} = (e^{\rho}-1)  - \frac{1}{V}\int_{X}(e^{\rho}-1) d\mu 
\]
which is the orthogonal projection in $L^{2}(d\mu)$ of $e^{\rho}-1$ to the complement of the kernel of $L_{\rho}$.  We have the following

\begin{prop}\label{prop: Lrho}
The following properties hold:
\begin{enumerate}
\item[(i)] $L_{\rho}$ is self-adjoint with respect to the $L^2$ inner-product induced by $d\mu$, and the kernel of $L_{\rho}$ consists of the constants.
\item[(ii)] Let $\nu_1$ be the first non-zero eigenvalue of $L_{\rho}$.  Then $\nu_1$ is negative, and $\nu_1 \leq -\lambda$.
\end{enumerate}
\end{prop}
\begin{proof}
The first point is easy.  For $(\mathrm{ii})$, if $\nu_1$ is the first non-zero eigenvalue, then applying $\nabla_{\bar {\ell}}$ to the equation $L_{\rho} f= \nu_1 f$ and commuting derivatives we get
\[
\begin{aligned}
\nu_1 \nabla_{\bar{\ell}} f&= g^{i\bar{j}} \nabla_i \nabla_{\bar j} \nabla_{\bar{\ell}} f -R_{\bar{\ell}}^{\bar{p}}\nabla_{\bar{p}}f + g^{i\bar{j}}\nabla_{\bar{j}} f \nabla_{\bar{\ell}}\nabla_i\rho + g^{i\bar{j}}\nabla_{\bar{\ell}}\nabla_{\bar{j}} f \nabla_i\rho\\
&= g^{i\bar{j}} \nabla_i \nabla_{\bar{j}} \nabla_{\bar{\ell}} f - \lambda \nabla_{\bar{\ell}} f +g^{i\bar{j}}\nabla_{\bar{\ell}}\nabla_{\bar{j}} f \nabla_i\rho\\
&= e^{-\rho}\nabla_{i}\left( e^{\rho} \nabla_{\bar{j}}\nabla_{\bar{\ell}} f \right) - \lambda \nabla_{\bar{\ell}} f.
\end{aligned}
\]
Multiplying by $g^{m \bar{\ell}}\nabla_m f d\mu$ and integrating by parts gives
\[
-\int_{X} |\bar{\nabla}\bar{\nabla}f|^{2} d\mu - \lambda \int_{X} |\bar{\nabla} f|^{2} d\mu = \nu_{1}\int_{X} |\bar{\nabla} f|^{2}d\mu
\]
and the proposition follows.
\end{proof}

An important application is
\begin{cor}\label{cor: DconvFlow}
Along the inverse Monge-Amp\`ere flow we have
\[
\frac{d^2}{dt^2} D(\phi) \geq 0, \qquad \ddt R(\phi) \leq 0.
\]
\end{cor}
\begin{proof}
Since $\ddt D(\phi) = -R(\phi)$, the two statements are equivalent.  Recall that we have shown that, along the $MA^{-1}$-flow we have
\[
\frac{d^2}{dt^2} D(\phi) = -\frac{2}{V} \int_{X} \left(L_{\rho} \tilde{f} +\lambda \tilde{f}\right) \tilde{f} e^{\rho} \omega_{\phi}^n
\]
where $f= e^{\rho}-1$ and
\[
\tilde{f} = f- \frac{1}{V} \int_{X} fe^{\rho}\omega^n.
\]
Let $\nu_1$ be the first non-zero eigenvalue of $L_{\rho}$.  By Proposition~\ref{prop: Lrho} (ii) we have that $\nu_1 \leq -\lambda$.  By its definition $\tilde{f}$ is orthogonal the constants in $L^{2}(X,d\mu)$, and so elliptic theory implies 
\[
-\int_{X} L_{\rho}\tilde{f} \cdot \tilde{f} d\mu \geq -\nu_1\int_{X} |\tilde{f}|^2 d\mu.
\]
Thus we have
\begin{equation}\label{eq: DingConvex}
\frac{V}{2}\frac{d^2}{dt^2} D(\phi) \geq -(\lambda+\nu_1)\int_{X} |\tilde{f}|^2 d\mu \geq 0.
\end{equation}
\end{proof}

We remark that the previous result follows formally from the description of the Ding functional as a Kempf-Ness functional.  Finally, we prove a general $C^{2}$ estimate along the flow, following the original computation of Yau \cite{Y} for the complex Monge-Amp\`ere equation, and its analog for the K\"ahler-Ricci flow \cite{Cao}.

\begin{prop}\label{prop: C2est}
There exists uniform constants $A, C$ depending only on a lower bound for the section curvature of $\omega_0$, and the dimension of $X$, so that on $X\times[0,T)$ we have
\[
\log \Tr_{\omega_0} \omega_{\phi} \leq A+C\left(\phi-\inf_{X\times [0,T)}\phi\right)  - \inf_{X\times [0,T)}\left(\lambda \phi + \rho + c(t)\right).
\]
In particular, $|\del\dbar \phi|_{\omega_0}$ is bounded uniformly on $X\times [0,T)$ in terms of $\sup_{[0,T)}\|\phi \|_{L^{\infty}}$ and $\inf_{X\times[0,T)} \rho$.
\end{prop}
\begin{proof}
The proof is by the maximum principle.  We begin by computing the evolution of $\Tr_{\omega_0} \omega$.  To ease notation, let us denote by $\hat{g}$ the K\"ahler metric associated  with $\omega_0$, and $g$ the K\"ahler metric associated with $\omega_{\phi}$.  The key computation is 
\[
\Delta_{g} \log \Tr_{\hat{g}} g \geq -B \Tr_{g}\hat{g} - \frac{\hat{g}^{j\bar{k}}R_{j\bar{k}}}{\Tr_{\hat{g}}g}
\]
where $B$ is a lower bound for the sectional curvature of the reference metric $\hat{g}$ (see, for example \cite[Lemma 3.7]{Sze14}). We compute
\[
\ddt \Tr_{\hat g} g = \Delta_{\hat{g}} \dot{\phi} = -\Delta_{\hat{g}} e^{\rho} = -e^{\rho} \left(\Delta_{\hat{g}} \rho + |\nabla \rho|^{2}_{\hat{g}}\right).
\]
From the definition of $\rho$ we obtain
\[
\ddt \Tr_{\hat{g}} g = e^{\rho} \left(-g^{j\bar{k}}R_{j\bar{k}} + \lambda \Tr_{\hat{g}}g - |\nabla \rho|^{2}_{\hat{g}} \right).
\]
We now apply the linearized operator and find
\[
\left(\ddt - e^{\rho} \Delta_{g}\right) \log \Tr_{\hat{g}} g \leq e^{\rho} \left( B \Tr_{g}\hat{g} + \lambda -\frac{ |\nabla \rho|^{2}}{\Tr_{\hat{g}}g} \right) \leq  e^{\rho} \left( B \Tr_{g}\hat{g} + 1\right).
\]
We now apply the maximum principle to the quantity $Q:= \log \Tr_{\hat{g}}g - C\phi$ for a large constant $C$.  We compute
\[
\left(\ddt - e^{\rho} \Delta_{g}\right) Q \leq e^{\rho} \left( B \Tr_{g}\hat{g} + 1\right) -C(1-e^{\rho}) + Ce^{\rho}(n- \Tr_{g}\hat{g}).
\]
Take $C=B+1 \geq 1$, so that
\[
\left(\ddt - e^{\rho} \Delta_{g}\right) Q \leq e^{\rho} \left( - \Tr_{g}\hat{g} + C(n+2)\right) -C \leq  e^{\rho} \left( - \Tr_{g}\hat{g} + C(n+2)\right).
\]
If $Q$ achieves its maximum at $(x^{*},t^{*}) \in X \times (0,T]$, then we have
\[
\Tr_{g}\hat{g} (x^{*},t^{*}) \leq C(n+2).
\]
On the other hand, we have
\[
\Tr_{\hat{g}}g \leq (\Tr_{g}\hat{g})^{n-1} \frac{\det g}{\det \hat{g}} = (\Tr_{g}\hat{g})^{n-1} e^{-(\rho +\lambda \phi +\rho_0 +c(t))}.
\]
Since $Q \leq Q(x^{*},t^*)$, some simple algebra yields the estimate.
\end{proof}

We end by recording the parabolic Evans-Krylov estimate, which allows us to obtain the higher order regularity of the flow.

\begin{prop}\label{prop: Krylov}
Let $\phi(t)$ evolve by the $MA^{-1}$-flow.  Suppose there exists a constant $A$ so that $A^{-1} \omega_0 \leq \omega_{\phi} \leq A \omega_0$ along the flow.  Then, for each $0<\alpha <1$, there exists a constant $C$, depending only on $\alpha, (M,\omega_0)$, $\sup_{X\times[0,T)} |\phi|$ and $A$ so that
\[
\| \phi(t) \|_{C^{2, \alpha}(X,\omega_0)} \leq C
\]
for all $t\in [0,T)$.
\end{prop}
\begin{proof}
The key point is that the $MA^{-1}$-flow is of the form
\[
\ddt \phi = F(\phi, \phi_{i\bar{j}})
\]
where $F(x,M)$ is concave in $M$, as a map from the positive definite Hermitian matrices to to $\mathbb{R}$.  Furthermore, the assumptions that $|\phi|$ is bounded and $A^{-1} \omega_0 \leq \omega_{\phi} \leq A \omega_0$ imply that the flow is uniformly parabolic.  With these observations, the proposition is nothing but the parabolic Evans-Krylov estimate \cite{Kry82}, with the subtlety that we only control the complex Hessian of $\phi$, rather than the full Hessian.  However, there is a now standard trick, going back to Wang \cite{W}, which allows us to apply the estimate of Krylov directly.  See, for example, \cite{TWWY, CJY}.
\end{proof}

Combining this proposition with the Schauder theory and a standard bootstrapping argument we obtain

\begin{cor}\label{cor: highEst}
Let $\phi(t)$ evolve by the inverse Monge-Amp\`ere flow.  Suppose there exists a constant $A$ so that $A^{-1} \omega_0 \leq \omega_{\phi} \leq A \omega_0$ along the flow.  Then, for each $\ell \in \mathbb{N}$, and $0<\alpha <1$, there exists a constant $C_{\ell}$, depending only on $\ell, \alpha, (M,\omega_0)$, $\sup_{X\times[0,T)} |\phi|$ and $A$ so that
\[
\| \phi(t) \|_{C^{\ell, \alpha}(X,\omega_0)} \leq C_{\ell}
\]
for all $t\in [0,T)$.
\end{cor}

\section{The inverse Monge-Amp\`ere flow on canonically polarized K\"ahler manifolds}\label{sec: DMflowCan}

In this section we consider the flow in the case that $\lambda =-1$.  The goal of is to establish the long-time existence of the flow and the convergence towards the unique K\"ahler-Einstein metric with negative Ricci curvature.  We begin with some easy estimates.

\begin{lem}\label{lem: rhobdcanon}
Along the $MA^{-1}$-flow with $\lambda =-1$ we have
\[
\inf_{X}\rho_0 \leq \rho \leq \max\{ \sup_{X} \rho_0, 1+\dot{c}(0)\}.
\]
\end{lem}
\begin{proof}
We compute the evolution equation for $e^{\rho}$.  We have
\[
\begin{aligned}
\ddt \rho &= -\Delta_{\omega} \dot{\phi} +\dot{\phi} + \dot{c}\\
&= \Delta_{\omega} e^{\rho} +1-e^{\rho} +\dot{c}
\end{aligned}
\]
and so
\[
\left(\ddt -e^{\rho} \Delta_{\omega} \right)e^{\rho} = -e^{2\rho}+e^{\rho}+\dot{c}e^{\rho}.
\]
By the normalization condition
\[
\int_{X}e^{\rho} \omega^n = \int_{X} \omega^{n}
\]
we see that $\inf_{X} e^{\rho} \leq 1 \leq \sup_{X}e^{\rho}$.  It follows that if $\rho$ achieves its minimum at $(x^*, t^{*}) \in(0,T]\times X$, then at this point we have
\[
\left(\ddt -e^{\rho} \Delta_{\omega} \right)e^{\rho} \geq \dot{c}e^{\rho}.
\]
On the other hand, $\dot{c} \geq 0$ by Lemma~\ref{lem: BasicProps}, and so by the strong maximum principle we conclude that $\rho$ achieves its minimum at $t=0$.  At the maximum of $e^{\rho}$ we obtain
\[
e^{\rho} \leq 1+\dot{c}.
\]
On the other hand, $\dot{c} = -\dot{D}$, and the Ding functional is convex along the flow.  Thus $\dot{c} \leq \dot{c}(0)$, and the result follows.
\end{proof}
We will use this to prove that the Ding functional is uniformly bounded from below along the flow.  We apply equation~\eqref{eq: DingConvex} in the case $\lambda=-1$, recalling that $\nu_1\leq 0$, to obtain
\[
\frac{d^2}{dt^2}D(\phi) \geq\frac{2}{V} \int_{X} |\tilde{f}|^2 e^{\rho} \omega^n.
\]
Now we need
\begin{prop}
Let $f=e^{\rho}-1$, and $\tilde{f} = f- \frac{1}{V} \int_{X}fd\mu$. There is a constant $\delta_0>0$ depending only on the initial data so that
\[
\int_{X}\tilde{f}^2e^{\rho}\omega^n \geq \delta_0 \int_{X} f^2 \omega^n + \frac{\delta_0}{V}\left( \int_{X} f^2\omega^n \right)^2.
\]
\end{prop}
\begin{proof}
The proof is straightforward.  To ease notation, let us again use the notation $d\mu =e^{\rho} \omega^n$.  By Lemma~\ref{lem: rhobdcanon} we have $d\mu \geq \delta_0 \omega^n$ so
\[
\begin{aligned}
\int_{X}\tilde{f}^2d\mu &\geq \delta_0 \int_{X}\tilde{f}^2 \omega^n\\
&= \delta_0 \int_{X} f^{2}\omega^n - 2\frac{\delta_{0}}{V} \left(\int_{X} f d\mu \right)\left(\int_{X} f \omega^n\right) + \frac{\delta_0}{V} \left( \int_{X}f d\mu\right)^2.
\end{aligned}
\]
On the other hand, by the normalization of $\rho$ we have $\int_{X}f \omega^n =0$.  This implies that the second term above is zero, and the third term can be written as
\[
\int_{X}fd\mu = \int_{X}f(f+1)\omega^n = \int_{X}f^2\omega^n.
\]
\end{proof}
Summarizing, we have proved that
\[
\frac{d^2}{dt^2} D(\phi) \geq \frac{2\delta_0}{V} \left( \int_{X} f^2 \omega^n + \frac{1}{V}\left( \int_{X} f^2\omega^n \right)^2\right).
\]
On the other hand, we have $\ddt D(\phi) = -\frac{1}{V}\int_{X}f^{2}\omega^n$.
\begin{prop}\label{prop: DingODE}
Along the $MA^{-1}$-flow, with $\lambda=-1$, there is a constant $\delta$ depending only on the initial data so that
\begin{equation}\label{eq: Dcanflow}
\frac{d^2}{dt^2} D(\phi) \geq 2\delta \left( -\ddt D(\phi) + \left( \ddt D(\phi)\right)^2\right).
\end{equation}
Furthermore, there exists a constants $B>0$, depending only on the initial data so that
\[
0 \geq \dot{D} \geq -Be^{-\delta t}.
\]
In particular, the Ding functional is uniformly bounded from below, and the constant $c(t)$ is uniformly bounded from above for as long as the flow exists.
\end{prop}

\begin{proof}
We have already proved equation~\eqref{eq: Dcanflow}, so it suffices to prove the lower boundedness.  This follows from analyzing the ODE for which $D$ is a super-solution.  The function
\[
A(t) := \frac{-1}{Ce^{2\delta t}-1}
\]
solves the ODE
\[
\ddt A = 2\delta( -A + A^2).
\]
By choosing $C >1$ sufficiently close to $1$, we can arrange that $\ddt D(\phi)|_{t=0} \geq A(0)$.  The comparison principle implies that $\ddt D(\phi) \geq A(t)$ for as long as the flow exists.  Since $A$ is integrable on $[0,\infty)$ we obtain the uniform lower bound for the Ding functional along the flow.  The upper bound for the constant $c(t)$ follows from Lemma~\ref{lem: BasicProps} (iii).
\end{proof}
An immediate corollary is that the $C^{0}$ norm of $\phi(t)$ is uniformly bounded along the flow.

\begin{cor}\label{cor: C0bdcanon}
There exists a uniform constant $C$, depending only on the initial data, so that
\[
|\phi(t)| \leq C
\]
holds along the flow.
\end{cor}
\begin{proof}
We first prove the lower bound.  Let $x_t \in X $ be the point where $\phi(t)$ achieves its infimum.  By Lemma~\ref{lem: rhobdcanon}, at the point $x_t$ we have
\[
\inf_{X} \rho_0 \leq \rho = \phi - \log \left( \frac{\omega_{\phi}^n}{\omega_{0}^{n}}\right) + \rho_0 +c(t) \leq \phi + \rho_0 +c(t)
\]
where we used that $\omega_\phi \geq \omega_0$ at $x_t$, the minimum of $\phi(t)$.  By Proposition~\ref{prop: DingODE} the constant $c(t)$ is uniformly bounded from above, so
\[
\phi(x_t) \geq -C
\]
for a uniform constant $C$.  For the upper bound, we use that $c(t)$ is increasing by Lemma~\ref{lem: BasicProps} $(iii)$.  From Jensen's inequality we have
\[
-c(0) \geq -c(t) = \log \left(\frac{1}{V} \int_{X}e^{\phi +\rho_0} \omega_0^n \right) \geq \frac{1}{V} \int_{X} (\phi +\rho_0) \omega_0^n
\]
so we conclude that $\int_{X}\phi \omega_0^n \leq C$ uniformly along the flow.  Since $\phi$ is $\omega_0$-plurisubharmonic, a standard argument involving the Green's function yields $
\phi(t) \leq C$ (see, e.g. \cite{Sze14}).
\end{proof}

We can now prove the main result of this section.

\begin{proof}[Proof of Theorem~\ref{thm: mainCanonPolar}]
Combining Corollary~\ref{cor: C0bdcanon} and Lemma~\ref{lem: rhobdcanon}, we see that $|\phi| + |\rho|$ is uniformly bounded along the flow.  By Proposition~\ref{prop: C2est} there is a constant $A$ so that $\omega_{\phi} \leq A \omega_0$ uniformly. Now we write
\[
-\log\left(\frac{\omega_\phi^{n}}{\omega_0^n}\right) = \rho -\phi -\rho_0 -c(t) \leq C
\]
using the bounds for $\phi, \rho$ and $c(t)$.  Since $\omega_{\phi}$ is bounded from above we get that $\omega_{\phi} \geq A^{-1}\omega_0$, after possibly increasing $A$.  By Proposition~\ref{prop: Krylov} and Corollary~\ref{cor: highEst}, we obtain uniform $C^{\ell,\alpha}$ bounds along the flow.  It follows that the flow exists for all time. 

Next we address the convergence of the flow.  Since we have obtained uniform $C^{\ell}$ bounds, the subsequential convergence is clear.  To obtain convergence of the $\phi(t)$ without subsequences, we note that for times $s>t$ we have
\[
\begin{aligned}
\int_{X}|\phi(t)-\phi(s)| \omega_0^{n} &\leq \int_{t}^{s}d\tau\int_{X}|\dot{\phi}(\tau)| \omega_{0}^{n}\\
&\leq \int_{t}^{s}d\tau \left( V\int_{X} (\dot{\phi}(\tau))^2 \omega_0^{n} \right)^{1/2} \\
&\leq C \int_{t}^{s}d\tau \left( -V^2\dot{D}(\tau) \right)^{1/2}
\end{aligned}
\]
where we used that $\omega_0, \omega_\phi$ are uniformly equivalent along the flow.  By Proposition~\ref{prop: DingODE}, $0\leq -\dot{D} \leq Be^{-2\delta t}$ for uniform constants $B, \delta >0$, and so
\[
\int_{X}|\phi(t)-\phi(s)| \omega_{0}^{n} \leq C\left( e^{-\delta t} -e^{-\delta s} \right)
\]
which implies that $\phi(t)$ is Cauchy in $L^{1}(X,\omega_0)$.  It follows that $\phi(t)$ converges in $L^{1}$ to a limit $\phi_{\infty}$.  Furthermore, from the uniform $C^{\ell}$ estimates it is easy to see that $\phi(t)$ converges to $\phi_{\infty}$ in $C^{\ell}$ for all $\ell \in \mathbb{N}$.  It remains only to show that $\phi_{\infty}$ is K\"ahler-Einstein.  This is clear, since the estimate
$0\leq -\dot{D} \leq e^{-2\delta t}$ implies that
\[
\int_{X}(1-e^{\rho_{\infty}})^{2}\omega_{\phi_{\infty}}^{n}=\lim_{t\rightarrow \infty}\int_{X}(1-e^{\rho})^{2}\omega_{\phi}^{n}= \lim_{t\rightarrow \infty} -V\dot{D}(t) =0.
\]
Since $\omega_{\infty}$ is a smooth K\"ahler metric, we deduce that $\rho_{\infty} =0$, and hence $\phi_{\infty}$ is K\"ahler-Einstein.
\end{proof}
 
\section{The inverse Monge-Amp\`ere flow on Fano manifolds}\label{sec: DMflowFano}

We now turn our attention to the Fano setting, when $\lambda =+1$.  The first major difficulty is that there is no estimate analogous to Lemma~\ref{lem: rhobdcanon}.  Indeed, the evolution equation for $\dot{\phi}$ reads
\[
\left(\ddt -e^{\rho}\Delta_{\omega}\right) \dot{\phi} = (1-\dot{\phi})(\dot{\phi}-\dot{c})
\]
and when $\lambda =1$, $\dot{c}$ is negative.  The reader can check that the ODE
\[
\ddt f = (1-f)(f+a),
\]
for $a>0$, for which $\inf_{X}\dot{\phi}$ is a supersolution, has solutions diverging to $-\infty$ in finite time.  It is therefore necessary to take a more delicate approach.  We begin with some easy observations.  

\begin{lem}\label{lem: Fanoeasybd}
Along the inverse Monge-Amp\`ere flow on a Fano manifold there are constants $A_1, A_2 >0$ depending only on the initial data so that
\[
\phi(t) \leq A_1+t \qquad \rho(t) \geq -t+c(t)-A_2.
\]
\end{lem}
\begin{proof}
The upper bound for $\phi$ follows from the evolution equation
\[
\dot{\phi} = 1-e^{\rho} \leq 1.
\]
For the lower bound of $\rho$ we compute
\[
\ddt \rho =\Delta_{\omega}e^{\rho} +e^{\rho}-1+\dot{c}.
\]
Thus $h(t) := \inf_{X}\rho(t)$ satisfies
\[
\ddt h \geq e^{h}-1+\dot{c} \geq -1+\dot{c}
\]
in the viscosity sense.  By the comparison principle we deduce $\inf_{X}\rho(t) \geq -t+c(t) -A_2$ for $A_2$ depending only on $\rho_0$.
\end{proof}

We now prove a lower bound for $\phi$ on $X\times[0,T)$ for $T<+\infty$.  As discussed at the beginning of this section, this does not follow easily from the evolution equation.  Instead,
we will use a compactness argument.  Recall that that the Aubin-Yau energy $E(\phi)$ is constant along the flow.  A standard computation shows that for an $\omega_{0}$-psh function $\psi$ we have
\[
E(\psi) = \frac{1}{(n+1)V} \sum_{j=0}^{n} \int_{X} \psi \, \omega_{\psi}^{j}\wedge \omega_0^{n-j}.
\]
Observe that if $\psi$ is negative, then we have
\[
-\frac{1}{V} \int_{X} \psi \omega_{\psi}^n \leq -(n+1) E(\psi).
\]
Let $\tilde{\phi}(t) = \phi(t)-\sup_{X}\phi(t)$, then along the $MA^{-1}$-flow we have
\[
\begin{aligned}
-\frac{1}{V}\int_{X}\tilde{\phi} \omega_{\tilde{\phi}}^n \leq -(n+1) E(\tilde{\phi}) &= -(n+1) E(\phi(t)) +(n+1) \sup_{X} \phi(t) \\
&\leq  -(n+1) E(\phi(0)) +C(t+1)
\end{aligned}
\]
where we used Lemma~\ref{lem: Fanoeasybd}.  Suppose that $T<+\infty$, and there exists a sequence of times $t_i \rightarrow T$ such that $\phi(t_i) \rightarrow-\infty$.  By Jensen's inequality we have
\[
c(t) \leq \frac{1}{V} \int_{X} (\phi(t)-\rho_0) \omega_0^n.
\]
Since $c(t)$ is convex, we have $c(t) \geq c(0) +t\dot{c}(0) >-\infty$, and so we deduce that $\sup_{X}\phi$ is uniformly bounded from below on $X\times[0,T)$.  Thus, up to passing to a subsequence we can assume that $\phi(t_i) \rightarrow \phi_T$ weakly, and in $L^{1}(X,\omega_0)$, where $\phi_T$ is some $\omega_0$-PSH function.  We need the following result

\begin{lem}[Guedj-Zeriahi]\label{lem: GZ}
Suppose that $\psi_j$ is a sequence of negative $\omega_0$-PSH functions on $X$ converging in $L^1$ to a $\omega_0$-PSH function $\psi_{\infty}$.  If there is a uniform constant $C$ such that
\[
-\int_{X}\psi_j \omega_{\psi_j}^n \leq C
\]
then $\psi_j$ lies in the energy space $\mathcal{E}^{1}(X,\omega_0)$, and in particular, $\psi_{\infty}$ has zero Lelong numbers.
\end{lem} 
This result can be deduced by combining \cite{GZ07} Corollary 1.8, and Corollary 2.7.  Since $\sup_{X}\phi(t)$ is uniformly bounded from above and below on $[0,T)$, Lemma~\ref{lem: GZ} implies that $\phi_T$ has zero Lelong numbers.  By Skoda's lemma (see, e.g. \cite[Corollary 3.2]{Zer01}) we have that
\[
\int_{X}e^{-p\phi_T} <+\infty
\]
for all $p>0$.  The effective version of Demailly-Koll\'ar's semi-continuity theorem for log-canonical thresholds \cite[Main Theorem 0.2, (2)]{DK01} (see also \cite{PS00}) says that
\[
e^{-p\phi(t_j)} \rightarrow e^{-p\phi_{T}} \quad \text{ in } L^{1}(X,\omega_0).
\]
Fixing $p=2$ we conclude that there is a uniform constant $C$ so that for all $j$ we have
\[
\int_{X}e^{-2\phi(t_j)} \omega_0^n \leq C.
\]
On the other hand, $\phi(t_j)$ solves the complex Monge-Amp\`ere equation
\[
\omega_{\phi(t_j)}^n = e^{-\phi(t_j)} \cdot e^{-\rho(t_j) +\rho_0 +c(t_j)} \cdot \omega_0^n.
\]
Recall that $\rho$ is uniformly bounded from below on $[0,T)$ thanks to Lemma~\ref{lem: Fanoeasybd}.  Since $e^{-\phi(t_j)}$ is uniformly bounded in $L^{2}$, this equation is of the form
\[
\omega_{\phi(t_j)}^n = e^{F} \cdot \omega_0^n.
\]
where $e^{F}$ is uniformly bounded in $L^2$.  We can therefore apply Kolodziej's $C^{0}$ estimate \cite{K03, K05}, or even the Alexandroff-Bakelman-Pucci estimate \cite{Blo05}, to deduce that
\[
\sup_{X}\phi(t_j) -\inf_{X}\phi(t_j) \leq C
\]
for a uniform constant $C$.  By the lower bound for $\sup_{\phi}$ we deduce that $\inf_{X}\phi(t_j) \geq -C >-\infty$ for a constant independent of $j$.  This is a contradiction.  We have proved
\begin{prop}\label{prop: Fanolowbd}
Suppose that the inverse Monge-Amp\`ere flow exists on $[0,T)$ for some $T<+\infty$.  There is a constant $C_{T} < +\infty$ such
\[
\inf_{X\times[0,T)} \phi(t) \geq -C_T.
\]
\end{prop}

Combining Proposition~\ref{prop: Fanolowbd} and Lemma~\ref{lem: Fanoeasybd} we have shown that if the $MA^{-1}$-flow exists on $[0,T)$ for $T<+\infty$, then there is a constant $C_T<+\infty$ so that
\[
|\phi(t)|_{L^{\infty}} -\inf_{X}\rho(t) \leq C_T.
\]
Thanks to Proposition~\ref{prop: C2est} we deduce that $\Delta_{\omega_0} \phi \leq C_{T}'$ on $[0,T)$.  The only remaining obstacle to obtaining the long-time existence of the flow is to prove
\begin{prop}
Suppose that the inverse Monge-Amp\`ere flow exists on $[0,T)$.  For every $T<+\infty$, there exists a constant $C_{T}< +\infty$ so that for all $t\in [0,T)$ we have
\[
C_{T}^{-1} \omega_0 \leq \omega_{\phi(t)} \leq C_{T} \omega_0.
\]
\end{prop}
\begin{proof}
As remarked above, we have already obtained the upper bound $\omega_{\phi} \leq C_{T} \omega_0$.  In order to obtain the lower bound we claim that it suffices to prove that $\rho \leq C_T$.  To see this, note that
\[
\rho = -\phi -\log\left(\frac{\omega_{\phi}^n}{\omega_0^n}\right) + \rho_0 +c(t).
\]
Thanks to the bounds for $\phi$, an upper bound for $\rho$ implies an estimate $\omega_{\phi}^{n} \geq \delta \omega_0^n$.  Combining this with the upper bound $\omega_{\phi} \leq C_{T} \omega_0$, we see that $\omega_{\phi} \geq C_{T}^{-1}\omega_0$, after possibly increasing $C_T$.  To estimate $\rho$ we will use the maximum principle.  First, note that since
\[
\|\phi\|_{L^{\infty}} +\| \Delta_{\omega_0} \phi\|_{L^{\infty}} \leq C,
\]
on $[0,T)$, elliptic regularity theory and the Sobolev Imbedding theorem implies that
\[
\sup_{X} |\nabla \phi|_{\omega_0} \leq C_1
\]
for a constant $C_1$ depending on $T$, and $(X,\omega_0)$. Fix $\epsilon >0$ to be determined and compute
\[
\left(\ddt -e^{\rho}\Delta_{\omega}\right) \rho -\epsilon \phi= (1+\epsilon)(e^{\rho}-1)+e^{\rho}|\nabla \rho|^{2}_{\omega_{\phi}} +\dot{c} +\epsilon e^{\rho}(n-\Tr_{\omega_\phi}\omega_0).
\]
Let $(x^{*},t^{*})$ be the point where $\rho-\epsilon \phi$ achieves it's maximum on $X\times[0,T)$.  Without loss of generality, we can assume $t^{*}>0$.  At $(x^{*},t^{*})$ we have
\[
\nabla \rho =\epsilon \nabla \phi
\]
and so
\[
|\nabla \rho|^{2}_{\omega_{\phi}} = \epsilon^{2}|\nabla \phi|^{2}_{\omega_{\phi}} \leq \epsilon^{2}\left(\Tr_{\omega_{\phi}} \omega_0\right) |\nabla \phi|^{2}_{\omega_0} \leq \epsilon^2C_1\Tr_{\omega_{\phi}} \omega_0.
\]
By the maximum principle, at $(x^*,t^*)$ we have
\[
\begin{aligned}
0&\leq (1+\epsilon)(e^{\rho}-1)+e^{\rho}|\nabla \rho|^{2}_{\omega_{\phi}} +\dot{c} +\epsilon e^{\rho}(n-\Tr_{\omega_\phi}\omega_0)\\
&\leq (1+\epsilon)(e^{\rho}-1) + \epsilon^{2}C_1\Tr_{\omega_{\phi}} \omega_0+\dot{c} +\epsilon e^{\rho}(n-\Tr_{\omega_\phi}\omega_0)
\end{aligned}
\]
choosing $\epsilon$ small so that $\epsilon -C_1\epsilon^{2} \geq \frac{\epsilon}{2}$ we have
\[
0 \leq  (1+\epsilon)(e^{\rho}-1) -\frac{\epsilon}{2}\Tr_{\omega_{\phi}} \omega_0+\dot{c} +\epsilon e^{\rho}n.
\]
Rearranging we obtain that at $(x^*,t^*)$ we have the estimate
\[
\Tr_{\omega_{\phi}} \omega_0 \leq \frac{2}{\epsilon}\left((1+\epsilon)(1-e^{-\rho}) + \epsilon n\right) \leq  \frac{2}{\epsilon}\left(1+(n+1)\epsilon\right).
\]
As discussed above this implies an upper bound for $\rho$ at $(x^*, t^*)$.  Thus we conclude
\[
\rho-\epsilon\phi \leq \rho(x^*, t^*)-\epsilon\phi (x^*,t^*) \leq C_T - \epsilon \inf_{X\times[0,T)} \phi
\]
and we deduce that $\rho \leq C_T +2\epsilon\left(\phi -  \inf_{X\times[0,T)} \phi\right)$, which finishes the proof.
\end{proof}

Combining the above estimates with Proposition~\ref{prop: Krylov} and Corollary~\ref{cor: highEst} we obtain

\begin{thm}
The inverse Monge-Amp\`ere flow on a Fano manifold exists for all time.
\end{thm}

We now turn our attention to the situation when $X$ is Fano, K\"ahler-Einstein, and has no holomorphic vector fields (or equivalently, ${\rm Aut}(X)$ is discrete). Our goal is to show that the $MA^{-1}$-flow converges.  

The Mabuchi functional, or $K$-energy, on the space $\cH$ is defined by its variation
\[
\delta M(\phi) = -\frac{1}{V}\int_{X}\delta\phi(R_{\omega_{\phi}} -n) \omega_{\phi}^n
\]
where $R_{\omega_{\phi}}$ is the scalar curvature of $\omega_{\phi}$.  This functional integrates to a well-defined functional on $\cH$, which can be written as \cite{Ch00, Tian00}  
\[
\begin{aligned}
M(\phi) &=\frac{1}{V}\left[\int_{X}\log\left(\frac{\omega_{\phi}^n}{\omega_0^n}\right) \omega_{\phi}^n - \sum_{j=0}^{n-1} \int_{X} \phi {\rm Ric}(\omega_0)\wedge\omega_{\phi}^{j}\wedge \omega_{0}^{n-1-j}\right]\\
 &+ \frac{1}{V}\left[ \frac{n}{n+1}\sum_{j=0}^{n} \int_{X} \phi \omega_{\phi}^{j}\wedge \omega_0^{n-j}\right].
 \end{aligned}
\]
A simple integration by parts shows that we can write
\begin{equation}\label{eq:MabFunc}
M(\phi) = D(\phi) + H_{\nu_{\phi}}\left(\frac{\omega_{\phi}^{n}}{V}\right) +\frac{n!}{V}\int_{X} \rho_0 \omega_0^n
\end{equation}
where
\[
H_{\nu_{\phi}}\left(V^{-1}\omega_{\phi}^{n}\right) := \int_{X} \log\left(\frac{V^{-1}\omega_{\phi}^{n}}{\nu_{\phi}}\right) \frac{\omega_{\phi}^n}{V},\qquad \nu_{\phi} = \frac{e^{-\phi +\rho_0}\omega_0^{n}}{\int_{X} e^{-\phi +\rho_0} \omega_0^n}
\]
is the relative entropy. We have the following lemma
\begin{lem}\label{lem: Mabdec}
Along the $MA^{-1}$-flow on a Fano manifold we have
\[
\ddt M(\phi) = -\int_{X}|\nabla \rho|^{2}\nu_\phi.
\]
In particular, along the $MA^{-1}$-flow we have $M(\phi(t))\leq M(\phi(0))$.
\end{lem}

\begin{proof}
By the definition of $\rho$, we have
\[
H_{\nu_{\phi}}\left(V^{-1}\omega_{\phi}^{n}\right)  = -\frac{1}{V}\int_{X}\rho \omega_{\phi}^{n}
\]
and so
\[
\begin{aligned}
\ddt H_{\nu_{\phi}}\left(V^{-1}\omega_{\phi}^{n}\right)  &= -\frac{1}{V}\int_{X}\left(\dot{\rho} +\rho \Delta_{\omega}\dot{\phi}\right) \omega_{\phi}^{n}\\
&= -\dot{c} +\frac{1}{V}\int_{X}\rho \Delta_{\omega_{\phi}}e^{\rho} \omega_{\phi}^n\\
&= -\dot{c} -\frac{1}{V}\int_{X} |\nabla \rho|^{2}e^{\rho} \omega_{\phi}^{n} = -\dot{c} -\int_{X}|\nabla \rho|^{2}\nu_{\phi}.
\end{aligned}
\]
At the same time, by Lemma~\ref{lem: BasicProps} (iii), we have that $\dot{D} = \dot{c}$ along the $MA^{-1}$-flow. The lemma follows.
\end{proof}

Assuming $X$ admits a K\"ahler-Einstein metric, and has no nontrivial holomorphic vector fields, a theorem of Phong-Song-Sturm-Weinkove \cite{PSSW08}  says that there are positive constants $A,B$ so that
\[
D(\phi) \geq AJ(\phi) -B, \quad \text{ where } \quad J(\phi) = \frac{1}{V}\int_{X}\phi \omega_0^n - E(\phi)
\]
and 
\[
M(\phi) \geq AJ(\phi)-B,
\]
in other words, the Ding, and Mabuchi functionals are coercive.  In fact, the coercivity of $M(\phi)$ follows from the coercivity of $D(\phi)$, since $M(\phi) \geq D(\phi) -C$ for a constant $C$ as easily follows from~\eqref{eq:MabFunc} and Jensen's inequality.

Using the coercivity of the Ding functional, together with the fact that $E(\phi)$ is constant along the flow, and the Ding functional is decreasing, we deduce a uniform upper bound for $\int_{X} \phi \omega_0^n$.  Since $\phi$ is $\omega_0$-PSH along the flow we obtain
\[
\phi(t) \leq C
\]
uniformly in time by a standard Green's function argument.  Next, since the Aubin-Yau energy is constant along the flow, arguing as in the proof of Proposition~\ref{prop: Fanolowbd}
\[
-\frac{1}{V}\int_{X}\phi \omega_\phi^n \leq -(n+1)E(\phi)+ n\sup_{X}\phi \leq C
\] 
and so $\sup\phi$ is also bounded from below along the flow.  Thus, we can take a subsequence $t_{k} \rightarrow +\infty$ such that  
\[
\phi(t_{k}) \rightarrow \phi_{\infty}
\]
weakly and in $L^1$ for some $\omega_0$-PSH function $\phi_{\infty}$ and $\phi_{\infty}$ has zero Lelong numbers.  By the Skoda's theorem and the semi-continuity theorem \cite{DK01} we know that $e^{-\phi(t_k)}$ converges to $e^{-\phi_{\infty}}$ in $L^{p}$ for all $p>0$.  It follows that we have convergence of probability measures
\[
\nu_{\phi(t_{k})} = \frac{e^{-\phi(t_k) +\rho_0}\omega_0^{n}}{\int_{X} e^{-\phi(t_k) +\rho_0} \omega_0^n}  \longrightarrow \nu_{\phi_{\infty}}.
\]
We now recall the definition of the strong topology.  Consider the set
\[
\mathcal{E}^{1}(X,\omega_0) := \left\{ \phi \in {\rm PSH}(X,\omega_0) : \int_{X} \langle \omega_\phi^n\rangle = V \text{ and } \int_{X} \phi \langle \omega_\phi^n\rangle <+\infty \right\}
\]
where $\langle \omega_\phi^n \rangle$ denotes the non-pluripolar product \cite{GZ07}.  Let $\mathcal{E}^{1}_{norm} \subset \mathcal{E}^{1}$ consist of those functions with $\sup_{X}\phi =0$. 

\begin{dfn}
The {\em strong topology} on $\mathcal{E}^1_{norm}$ is the coarsest refinement of the weak topology on $\mathcal{E}^{1}_{norm}$, such that the map
\[
\phi \longrightarrow J(\phi)
\]
(which is defined by the non-pluripolar product) becomes continuous.
\end{dfn}

See \cite{BBEGZ11}, and in particular, Proposition 2.6.
It follows that the $\phi \mapsto \langle \omega_\phi^n\rangle$ is continuous in the strong topology. 
Since the Mabuchi functional is coercive, and $M(\phi(t)) \leq M(\phi(0))$ by Lemma~\ref{lem:  Mabdec}, \cite[Theorem 4.14]{BBEGZ11} says that $\phi(t_k) \rightarrow \phi_{\infty}$ in the strong topology, and hence  $\omega_{\phi(t_k)} \rightarrow \omega_{\phi_{\infty}}^n$ as measures, where the latter measure is interpreted in the non-pluripolar sense \cite{GZ07}.  We claim that
\begin{equation}\label{eq: measureEq}
\frac{1}{V} \omega_{\phi_{\infty}} = \nu_{\phi_{\infty}}.
\end{equation}
Let us assume this for the moment and finish the proof.  By definition, $\nu_{\phi_{\infty}} = e^{F} \omega_{0}^{n}$ for some $e^{F} \in L^{p}$ for all $p>1$, and so \cite[Theorem A]{EGZ09} implies that $\phi_{\infty}$ is bounded.  The regularity theorem of Sz\'ekelyhidi-Tosatti \cite{ST09} implies that $\phi_{\infty}$ is smooth.  It follows immediately that $\phi_{\infty}$ is K\"ahler-Einstein.  It remains to prove~\eqref{eq: measureEq}.    Since the Ding functional is convex along the flow, and bounded from below we have that
\[
\lim_{t\rightarrow \infty} \ddt D(\phi) = \lim_{t\rightarrow \infty} \frac{-1}{V} \int_{X}(e^{\rho}-1)^{2}\omega_{\phi}^{n}=0
\]
and so
\[
1 = \lim_{t\rightarrow \infty} \frac{1}{V}\int_{X}e^{2\rho} \omega_{\phi}^{n} = \lim_{t\rightarrow \infty} \int_{X} e^{\rho} \nu_{\phi}
\]
where we used the definition of $\rho$ and $\nu_{\phi}$.  Recall the relative entropy
\[
H_{V^{-1}\omega_{\phi}^{n}}\left(\nu_{\phi}\right) := \int_{X} \log \left(\frac{\nu_{\phi}}{V^{-1}\omega_{\phi}^{n}}\right) \nu_{\phi} = \int_{X}\rho \nu_{\phi}. 
\]
 By Jensen's inequality we have
\[
0\leq  H_{V^{-1}\omega_{\phi}^{n}}\left(\nu_{\phi}\right) \leq \log \left( \int_{X}e^{\rho} \nu_\phi \right)
\]
and so $H_{V^{-1}\omega_{\phi}^{n}}\left(\nu_{\phi}\right) \rightarrow 0$ as $t\rightarrow +\infty$.  Now we apply Pinsker's inequality to conclude that
\[
\left| \int_{U} \frac{\omega_{\phi}^n}{V} - \int_{U}\nu_{\phi} \right| \rightarrow 0
\]
for every measurable set $U \subset X$.  This proves~\eqref{eq: measureEq}.  To finish the proof of Theorem~\ref{thm: mainFano} it only remains to note that the argument above applies for any sequence of times $t_i \rightarrow +\infty$.  The K\"ahler-Einstein metric is unique by \cite{BM85} and so it follows that $\phi(t) \rightarrow \phi_{\infty}$ without taking subsequences.  Furthermore, by the semi-continuity theorem \cite{DK01} we see that $\phi(t) \rightarrow \phi_{\infty}$ in $L^{p}$ for all $p$.

\begin{rk}
We expect the convergence of $\phi(t)$ to the K\"ahler-Einstein metric is in the $C^{\infty}$ topology.  In the setting of the K\"ahler-Ricci flow the smooth convergence of the flow on K\"ahler-Einstein Fano manifolds is an unpublished theorem of Perelman, which was proved by the first author and Sz\'ekelyhidi \cite{CSz} (see also \cite{TZZZ}).  The proof makes fundamental use of Perelman's estimates for the Ricci flow \cite{ST08}, and in particular, the uniform bounds for the Ricci potential $\rho$ along the flow.  We note that uniform bounds for $\rho$ along the inverse Monge-Amp\`ere flow would easily imply the smooth convergence of the flow.  However, as we will see in Section~\ref{sec: toric}, the Ricci potential diverges in general.
\end{rk}

We now study the convergence of the flow when $X$ admits holomorphic vector fields.  We begin by proving

\begin{prop}\label{prop: EVI}
Let $\phi(t)$ be a solution of the $MA^{-1}$ flow.  Then we have
\[
\lim_{t\rightarrow \infty} D(\phi(t)) = \inf_{\psi \in \cH} D(\psi).
\]
\end{prop}

\begin{proof}
The proof follows a general strategy for gradient flows which applies quite generally, and was exploited, for example, by He in the study of the Calabi flow \cite{He13}.  Fix any $\psi \in \cH$, and, for each $t\in \mathbb{R}$. By \cite{Ch00a} we can take $\Phi_t(s)$ as the unique $C^{1,\alpha}$ geodesic in $\cH$ with $\Phi_t(0) = \phi(t)$, and $\Phi_t(1)=\psi$.  Let $d(\psi, \phi(t))$ denote the distance in $\cH$.  By \cite[Theorem 6]{Ch00a} we have
\[
\ddt d(\psi, \phi(t)) = -d(\psi, \phi(t))^{-\frac{1}{2}} \int_{X}\left(\frac{d\Phi_t}{ds}\bigg|_{s=0}\right) \ddt \phi(t)\, \omega_{\phi(t)}^{n}.
\]
It is well-known that $D(\Phi(s))$ is $C^{1}$ in $s$, and convex \cite{Bern11} and hence
\[
\begin{aligned}
D(\psi) - D(\phi(t)) = \int_{0}^{1} \frac{d}{ds} D(\Phi(s)) ds &\geq  \frac{d}{ds} D(\Phi(s))\bigg|_{s=0}\\
& = \frac{1}{V} \int_{X} \frac{d}{ds} \Phi_t(0)\left( e^{\rho_{\phi(t)}} -1 \right) \omega_{\phi(t)}^{n}
\end{aligned}
\]
We therefore obtain that, along the $MA^{-1}$-flow we have
\[
\begin{aligned}
V\left[D(\psi)-D(\phi(t))\right] - d(\psi, \phi(t))^{\frac{1}{2}}\ddt d(\psi, \phi(t)) &\geq \int_{X} \frac{d\Phi_t}{ds}\bigg|_{s=0} \left( e^{\rho_{\phi(t)}} -1 + \ddt \phi\right) \omega_{\phi(t)}^{n}\\
&\geq 0
\end{aligned}
\]
which is an example of an {\em evolution variation inequality}.  Using the monotonicity of the Ding functional along the flow we have
\[
\begin{aligned}
t\left[D(\psi) - D(\phi(t))\right] &\geq \int_{0}^{t} \left[ D(\psi) -D(\phi_{\lambda})\right] d\lambda\\
&\geq \frac{1}{2V} \left[ d^2(\psi, \phi(t)) - d^2(\psi, \phi_0)\right].
\end{aligned}
\]
Since the right hand side is bounded from below (depending on $\psi)$ we deduce that $\lim_{t\rightarrow \infty} D(\phi(t)) \leq D(\psi)$, and the proposition follows.
\end{proof}

We can now prove

\begin{thm}
Suppose $X$ is Fano and admits a K\"ahler-Einstein metric.  Then the inverse Monge-Amp\`ere flow converges to a K\"ahler-Einstein metric modulo the action of $\Aut_0(X)$.
\end{thm}
\begin{proof}
Since $X$ admits a K\"ahler-Einstein metric, work of Li-Zhou \cite{LZ} (building on \cite{DR15}) shows that there are constants $\epsilon, C >0 $ so that
\[
D(\omega) \geq \epsilon \inf_{F\in \Aut_0(X)} J(F^{*}\omega) - C.
\]
Furthermore, the vanishing of the Futaki invariant implies that the Ding functional is invariant under the action of $\Aut_0(X)$.  Let $\omega_{t}$ be a solution of the $MA^{-1}$-flow.  By Proposition~\ref{prop: EVI}, for each $t$ we can find an element $F_t \in \Aut_0(X)$ so that $\tilde{\omega}_t := F_t^{*}\omega_t$ satisfies
\[
\lim_{t\rightarrow \infty} D(\tilde{\omega}_t) = \inf_{\psi \in \cH} D(\psi) \qquad J(\tilde{\omega}_t) \leq C.
\]
The remainder of the argument follows as in \cite[Theorem D]{BBGZ13}; we only sketch the details.  By \cite[Lemma 3.3]{BBGZ13} the metrics $\tilde{\omega}_t$ converge to a current $\tilde{\omega}_{\infty}$ with potential in $\cE^1$.  The Ding functional is lower semi-continuous along this sequence and so
\[
D(\tilde{\omega}_{\infty}) \leq \lim_{t\rightarrow \infty} D(\omega_t) = \inf_{\psi \in \cH} D(\psi).
\]
That $\tilde{\omega}_{\infty}$ is K\"ahler-Einstein follows from \cite[Theorem 6.6]{BBGZ13}.
\end{proof}

\begin{rk}
We remark that in the above arguments only the coercivity of the Ding functional was used, and not the existence of a K\"ahler-Einstein metric.  In particular, we obtain a new proof of the existence part of \cite[Theorem 1.6]{Tian97}.
\end{rk}

\section{The Ding functional, stability and Mabuchi solitons}\label{sec: stab}

Before considering in detail the toric case, we discuss some general properties of the Ding functional, which elucidate why the $MA^{-1}$-flow is a natural approach to producing optimally destabilizing test configurations.

\subsection{Non-Archimedean Ding functional}\label{modified NA Ding}
We start by reviewing the algebraic description of the slope of the Ding functional for Fano manifolds; we refer the reader to \cite{Berm16, BHJ15, BHJ16} for details. 
From a variational point of view, the Yau-Tian-Donaldson conjecture predicts that the boundary of the space $\cH$ should contain, as a dense subset, potentials coming from certain algebraic degenerations called test configurations.
\begin{dfn} 
Let $(X, L)$ be a polarized manifold. A test configuration is a flat family $\pi: (\cX, \cL) \to \C$ of $\Q$-polarized schemes, with $\cX$ normal, endowed with a $\C^*$-action $\la: \C^* \to \Aut(\cX, \cL)$ such that 
\begin{itemize}
\item[$(1)$] the projection morphism $\pi$ is equivariant, and 
\item[$(2)$] the generic fiber is isomorphic to $(X, L)$.  
\end{itemize} 
\end{dfn} 

\begin{rk}
We will say that a test configuration $(\cX, \cL)$ is a special degeneration if, in addition, the central fiber $\mathcal{X}_0$ is a $\mathbb{Q}$-Fano variety.
\end{rk}

We identify the fiber $(\cX_1, \cL_1) = \pi^{-1}(1)$ with $(X, L)$. In particular, any $1$ parameter subgroup $\mu: \C^* \to \Aut(X, L)$ generates a test configuration with $(\cX, \cL)$ the product family.  Since total space $\cX$ is normal, the canonical divisor $K_\cX$ and the associated numerical invariants are well-defined. The normality is even necessary to exclude pathological examples found in \cite{LX11}. Presently we are interested in the case $L=-K_X$. Note that even in this special case $\cL$ is not necessarily linearly equivalent to $-K_{\cX/\C}$ and this causes the difference between the Ding invariant (explained below) and the Donaldson-Futaki invariant of $(\cX,\cL)$. 

Following \cite{BHJ15} we introduce the algebraic counterpart of the classical energies in terms of the positivity of $\cL$. 
We take the standard compactification $\bar{\pi}: (\bar{\cX}, \bar{\cL})\to \P^1$ of $(\cX, \cL)$ so that $\cX$ and $\cL$ are the trivial family around $\infty \in \P^1$. This compactification is unique by equivariance and can be constructed by gluing $\cX$ with the trivial family. 

\begin{dfn}
We define 
the non-Archimedean Aubin-Yau energy by the self-intersection number: 
\begin{equation*}
E^\NA(\cX, \cL) := \frac{\bar{\cL}^{n+1}}{V(n+1)}. 
\end{equation*}
We define $L^\NA(\cX, \cL)$ in terms of the log canonical threshold by 
\begin{equation*}
L^\NA(\cX, \cL):=  \lct_{(\cX, \cB)}(\cX_0) -1, 
\end{equation*}
where the boundary divisor is uniquely determined by $\cB \sim_\Q -K_{\bar{\cX}/\C}-\bar{\cL}$ and $\supp{\cB} \subseteq \cX_0$. 
\end{dfn} 
These non-Archimedean functionals correspond to the classical energy functionals
\begin{equation*}
E(\phi)= \frac{1}{V(n+1)}\sum_{i=0}^n \int_X \phi \omega_\phi^i \wedge\omega_0^{n-i}, \ \ \ 
L(\phi) = -\log \int_X e^{-\phi+\rho_0} \omega_0^n, \ 
\end{equation*}
and in this notation we have
\begin{equation*}
D(\phi) = L(\phi) -E(\phi). 
\end{equation*}

Let $(\cX, \cL)$ be a test configuration, and fix a metric $H$ on $\mathcal{L}$ with positive curvature.  Let $h_0$ be a metric on $L$ with curvature $\omega_0$.  Then $(\cX,\cL)$ induces a ray $\phi(t) \in \cH$ by setting
\[
h_0e^{-\phi(t)}(\xi) = H(\la(e^{-t})\xi)
\]
for any $\xi \in L_{x}$.  It turns out that, in this situation, the function $t\mapsto D(\phi(t))$ is convex \cite{Bern11}, and hence the limit slope
\[
\lim_{t\rightarrow \infty} \frac{D(\phi(t))}{t}
\]
exists.  Since the Ding functional has a critical point at a K\"ahler-Einstein metric, we must have that $\lim_{t\rightarrow \infty} D(\phi(t))/t \geq 0$.  The relation with the non-Archimedean functionals is the following theorem of Berman \cite{Berm16}

\begin{thm}[Berman]\label{Ding energy vs Ding invariant}
Let $(\cX, \cL)$ be a test configuration for $(X, -K_X)$. 
Then for any ray $\phi^t$ induced by $(\cX, \cL)$, and a positive curvature fiber metric $H$ on $\cL$ we have the slope formula 
\begin{equation*}
 \lim_{t \to \infty} D(\phi^t)/t = D^\NA(\cX, \cL):= L^\NA(\cX, \cL) -E^\NA(\cX, \cL). 
\end{equation*}
\end{thm}

\begin{exam}[Product test configuration]
Let $\mu: \C^* \to \Aut(X)$ be a one-parameter subgroup. Define
\[
(\cX,\cL, \pi) = ( X\times \mathbb{C},\,\, p_{X}^{*}K_{X}^{-1},\,\, p_{\mathbb{C}})
\]
where $p_{X}, p_{\mathbb{C}}$ denote the projections from $X\times \mathbb{C}$ to $X$ and $\mathbb{C}$ respectively.  Define a $\mathbb{C}^{*}$ action by
\[
\mu(\alpha) (x,\tau) =(\mu(\alpha)x, \alpha \tau). 
\]
The generating holomorphic vector field of $\mu$ is uniquely determined by the property $\mu(e^{-t}) = \exp(t{\rm Re}(v))$ for $t\in \mathbb{R}$. Then we have
\[
D^\NA(\cX, \cL) = {\rm Fut}(\cX, \cL) := \frac{1}{V}\int_{X} (\Re(v)\rho_{\phi})\omega_{\phi}^{n},
\]
which is the classical Futaki invariant \cite{Fut}.
\end{exam}

\begin{dfn}
 We say that a Fano manifold is Ding-semistable if $D^\NA(\cX, \cL) \geq 0$ for any test configuration.  We say that $X$ is Ding-polystable if, in addition, $D^{NA}(\cX,\cL) =0$ if and only if $(\cX, \cL)$ is a product test configuration.
\end{dfn}

In fact, Berman's result Theorem~\ref{Ding energy vs Ding invariant} also shows that when $(\cX,\cL)$ is a special degeneration, then $D^{\NA}(\cX,\cL)$ is equal to the Donaldson-Futaki invariant ${\rm Fut}(\cX,\cL)$ and hence by Chen-Donaldson-Sun \cite{CDS15c}, Ding-polystability is equivalent to the existence of a K\"ahler-Einstein metric.

\subsection{ Relative stability and Mabuchi solitons}\label{sec: stable}

In the setting of the $MA^{-1}$-flow there is a natural notion of soliton, and of relative Ding stability.  Let $\fh$ denote the Lie algebra of $\Aut(X)$.  In other words, $\fh$ is the complex lie algebra of holomorphic vector fields.  Since $X$ is Fano, for any $v\in \fh$ and any K\"ahler metric $\omega$ we can find a complex valued function $f$, unique up to addition of a constant, so that
\[
{\rm grad}^{\mathbb{C}}_{\omega} f= v \quad \text{ or equivalently } \quad \iota_v \omega = \sqrt{-1}\,\dbar f. 
\]
Thus we may identify
\begin{equation}\label{eq: LieAlg}
\fh \sim \left\{ f: X\rightarrow \mathbb{C} : {\rm grad}^{\mathbb{C}}_{\omega} f \in \fh, \,\, \int_{X}f \omega^n=0 \right\} =: \fh_{\omega}
\end{equation}
and the set on the right becomes a Lie algebra under the Poisson bracket. Following Mabuchi \cite{Mab01} we define
\begin{dfn}
A K\"ahler metric $\omega_{\phi} \in 2 \pi c_1(X)$ is called a Mabuchi soliton if $e^{\rho_{\phi}}-1 \in \fh_{\omega_\phi}$. 
\end{dfn}
To relate this to the variational point of view, recall that by Poposition~\ref{prop: RCvar}, $\phi \in \cH$ is a critical point of the Ricci-Calabi energy if and only if
\[
\int_{X} \delta{\phi} (L_{\rho} +1) \tilde{f}\,\, d\mu =0, \quad \tilde{f} = e^{\rho} - \frac{1}{V}\int_{X}e^{\rho}d\mu
\]
where $d\mu = e^{\rho}\omega_{\phi}^n$.  In particular, $\phi$ is a critical point of the Ricci-Calabi energy if and only if $\tilde{f}$ lies in the kernel of $L_{\rho}+1$, and then a well-known application of the Bochner-Kodaira formula shows that $e^{\rho}-1 \in \fh_{\omega_\phi}$.  In this case it is easy to show that $e^{\rho}-1$ generates a self-similar solution of the $MA^{-1}$-flow.  

A soliton can equally be characterized as the critical point of the {\em modified} Ding energy, and this gives rise to a notion of relative stability, which we now describe.  Fix a  complex torus $T \subset \Aut(X)$.  Intrinsically, we should take $T$ to be the center of the reductive part of $\Aut(X)$, and in fact, this is well-defined whenever the Ricci-Calabi energy has a critical point \cite{Wan04}.  In any event, fix a torus $T \subset \Aut(X)$, let $\ft \subset \fh$ be its Lie algebra. By taking the generator we may embed the lattice $\Hom(\C^*, T)$ consists of one-parameter subgroups to $\ft$. The image of $\Hom(\C^*, T)\otimes \R$ is naturally identified with the Lie algebera $\fs$ of the maximal compact subgroup $S:=\Hom(\C^*, T)\otimes \S^1 \subset T$. 
Fix an $S$-invariant reference metric $\omega \in 2 \pi c_1(X)$ and let $\fs_\omega \subset \fg_{\omega}$ be the image of $\fs$ under the identification~\eqref{eq: LieAlg} induced by $\omega$.  Note that $\mu \in \Hom(\C^*, T)$ maps to $h_\mu \in \fs_\omega$ is nothing but the Hamiltonian function for the corresponding vector field. 
Following Futaki-Mabuchi \cite{FM95} we define an inner product on $\fs$ by
\[
\langle \mu_1, \mu_2 \rangle:= \frac{1}{V} \int_{X} h_{\mu_1}h_{\mu_2} \omega^n
\]
and this inner product is independent of the choice of $\omega$.  Let $L^{2}(X,\omega,\mathbb{R})_{0}$ denote the space of functions which are square integrable with respect to $\omega^n$, and have average $0$.  We let
\[
P_{\omega} :L^{2}(X, \omega, \mathbb{R})_{0} \rightarrow \fs_\omega
\]
be the orthogonal projection.  Fix an $S$-invariant reference K\"ahler form $\omega_0$, and denote by $\cH^{S}$ be the space of $S$-invariant K\"ahler potentials.  We define the relative Ding function $D_{T}: \cH^{S} \rightarrow \mathbb{R}$ by its variation
\[
\delta D_{T}(\phi) = \int_{X} (\delta \phi)P_{\omega_{\phi}}^{\perp}(e^{\rho}-1) \omega_{\phi}^{n}.
\]
Denote by $h_{\eta} := P_{\omega_{\phi}}(e^{\rho_{\phi}}-1)$.  Mabuchi \cite{Mab01} showed that $h_{\eta}$ induces an element $\eta \in \fs$ which is independent of $\omega_{\phi}$, and generates a $1$-parameter subgroup in $\Aut(X)$.  Furthermore, Mabuchi showed that $\eta$ is equivalent to the extremal vector field in the sense of Calabi, and so we will refer to $\eta$ as the {\em extremal} vector field.

For any $\mu \in \fs$, we get a Hamiltonian function $h_\mu$, and we define the $\mu$-modified Aubin-Yau energy by
 \[
 \delta E_{\mu} = \frac{1}{V} \int_{X} \delta\phi h_{\mu} \omega_{\phi}^{n}.
 \]
By \cite[Lemma 2.14]{BWN14} this integrates to a well-defined function on $\cH^{S}_{\omega_0}$ which is convex along curves in $\cH^{S}$ generated by test-configurations \cite[Proposition 2.17]{BWN14}.  In this case we can write the modified Ding functional as 
\[
D_{T}(\phi) = D(\phi) -E_{\eta}(\phi).
\]
The second author computed the slope of $E_{\mu}$ along any $T$-equivariant test configuration.

\begin{lem}[\cite{His16a}]\label{lem: inprtc}
Let $(\cX, \cL)$ be a $T$-equivariant test configuration endowed with the action $\lambda: \C^* \to \Aut(\cX, \cL)$ and fix a one-parameter subgroup $\mu$. 
Let $N_{k} = \dim_{\mathbb{C}}H^0(\cX_0, \cL_0^{\otimes k})$, and denote the $\C^*$-weights induced by $\lambda, \mu$ on $H^0(\cX_0, \cL_0^{\otimes k})$ as $\la_1, \dots \la_{N_k}$ and $\mu_1, \dots \mu_{N_k}$ respectively. Then the slope of $E_\mu$ is: 
\begin{equation*}
 \langle (\cX, \cL), \mu \rangle := \lim_{k\to \infty} \frac{1}{N_kk^2}\sum_{i=1}^{N_k} \la_i \mu_i.    
\end{equation*}
\end{lem}

When $(\cX,\cL)$ is a product test configuration generated by $\mu'$, then the limit on the right hand side in the above lemma computes $\langle \mu', \mu \rangle$ (see, e.g. \cite[Proposition 7.16]{Sze14}).  In particular, we can view Lemma~\ref{lem: inprtc} as extending the Futaki-Mabuchi inner product to general test configurations. 

In terms of this algebraic data, the extremal vector field $\eta$ is algebraically characterized by the property that   
\begin{equation}\label{extremal vector} 
D^\NA (\mu) - \langle \mu, \eta\rangle =0 
\end{equation} 
holds for any $\mu \in \fs$.  Here we have abusively written $D^\NA (\mu)$ for the Ding invariant of the product test configuration generated by $\mu$.
In addition, since the Ding invariant is equivalent to the Donaldson-Futaki invariant for any product test configuration, \eqref{extremal vector} gives another proof that $\eta$ is the same as the extremal vector field in the sense of Calabi.    
\begin{dfn} 
We say that a Fano manifold is relatively Ding-semistable (with respect to $T$) if the modified non-Archimedean Ding functional 
\begin{equation*}
D_T^\NA(\cX, \cL):= D^\NA(\cX, \cL) -\langle (\cX, \cL), \eta\rangle 
\end{equation*}
is non-negative for any $T$-equivariant test configuration. We say that $X$ is relatively Ding-stable if, in addition, $D_{T}^{\NA}(\cX,\cL) =0$ if and only if $(\cX, \cL)$ is a product test configuration.
\end{dfn} 
It is expected that relative Ding stability is equivalent to the existence of a Mabuchi soliton.  We partially confirm one direction of this correspondence.

\begin{thm}
If a Fano manifold admits a Mabuchi soliton, then it is relatively Ding-semistable. One can further show that $X$ is relatively K-stable. 
\end{thm} 
\begin{proof}
This is precisely the relative version of \cite{Berm16} and one can use the same argument.  
Actually, by the complex Pr\'ekopa's inequality of Berndtsson \cite{Bern11} the modified Ding energy is convex along any ray $\phi^t$ associated with a test configuration $(\cX, \cL)$. 
We compute the differential 
\begin{equation*}
\lim_{t\to0} D_T(\phi^t)/t
= \int_X \bigg(\frac{d}{dt}\bigg\vert_{t=0} \phi^t \bigg) (e^{\rho_0} -1-h_\eta) \omega_0^n. 
\end{equation*}
The right-hand side is zero if we assume that $\phi^0$ is a Mabuchi soliton. 
Therefore by the slope formula the semistability 
\begin{equation*}
D_T^\NA(\cX, \cL) 
= \lim_{t \to \infty} D_T(\phi^t)/t
\geq \lim_{t\to0} D_T(\phi^t)/t=0 
\end{equation*}
follows. Using the same argument as subsection 3.5 in \cite{Berm16} one can further show relative K-stability. 

\end{proof}

We end this section by noting the Ricci-Calabi version of Donaldson's lower bound estimate of the Calabi functional \cite{Don05}. 
Let us define the $L^2$-norm of a test configuration by 
\begin{equation*}
\norm{(\cX, \cL)}^2 :=  \lim_{k\to \infty} \frac{1}{N_kk^2}\sum_{i=1}^{N_k} (\la_i -\hat{\la})^2,  
\end{equation*}
where $\hat{\la}$ denotes the mean value of $\C^*$-weights (cf. \cite{Sze14}). 

\begin{thm}[\cite{His12}]
For a Fano manifold  we have the inequality: 
\begin{equation}\label{lower bound of Ricci-Calabi 2}
\inf_\phi R(\phi)^{\frac{1}{2}} \geq \sup_{(\cX, \cL)}\frac{-D^\NA(\cX, \cL)}{\norm{(\cX, \cL)}}. 
\end{equation}
\end{thm} 

We remark that if one varies the Ding functional over $\mu \in \fs$ we obtain the variation formula 
\begin{equation*}
\d \bigg( \frac{D^\NA(\mu)}{\norm{\mu}} \bigg)
= \frac{1}{\norm{\mu}} \bigg( D^\NA (\d \mu) -\frac{\langle \d\mu, \mu \rangle}{\langle \mu, \mu \rangle}D^\NA(\mu)\bigg) 
\end{equation*}
and hence the extremal vector field $\eta$ is characterized as the optimizer for the normalized Ding invariant over $\fs$. We expect that the Mabuchi flow achieves the equality in ($\ref{lower bound of Ricci-Calabi 2}$) in general.  We will prove this in the toric case in the next section.

\section{$L^2$ Convergence of the inverse Monge-Amp\`ere flow to the optimal destabilizer on toric Fano manifolds}\label{sec: toric}
Let $X$ be a toric Fano manifold of dimension $n$ polarized by $-K_{X}$.  $X$ being toric means that there is a torus $T \subset \Aut(X)$ of dimension $n$.  Let $S \subset T$ be the maximal compact subgroup.  Note that there is a natural lattice $N \subset \fs$ generated by the $1$-parameter subgroups ${\rm Hom}(S^1, S)$, and a dual lattice $M\subset \fs^{*}$.   If $\omega_0$ is an $S$-invariant K\"ahler form in the class $2 \pi c_1(X)$, then we get a moment map
\[
m: X \rightarrow \fs^{*}
\]
whose image $P$ is the associated moment polytope.  In our present setting, $P \subset \fs^{*}$ is convex polytope which is Delzant, reflexive and $0\in P$ is the only interior lattice point.  Let
\[
V_P := \int_{P} 1
\]
where, from now on, all unadorned integrals on $P$ are taken with respect to the Lebesgue measure. Any $\mu \in \fs$  is naturally identified with a linear function $a_{\mu}$ on $\fs^{*}$ by
\[
a_{\mu} := \langle \mu, \cdot \rangle - \frac{1}{V_P}\int_{P}\langle \mu, \cdot \rangle,
\]
where $\langle, \rangle$ is the natural pairing between $\fs, \fs^{*}$.  With this identification, the Futaki-Mabuchi inner product is given by
\[
\langle \mu_1, \mu_2 \rangle = \frac{1}{V} \int_{P} a_{\mu_1} a_{\mu_2}.
\]
According to Donaldson \cite{Don02}, $T$-equivariant test configurations of $(X,-K_{X})$ correspond to convex, rational, piecewise linear functions $f:P \rightarrow \mathbb{R}$.  For such a function Lemma~\ref{lem: inprtc} (see \cite{ZZ08}) implies that
\[
\langle (\cX,\cL), \mu \rangle = -\frac{1}{V_P} \int_{P}fa_{\mu}, \qquad \| (\cX, \cL)\|^2 = \frac{1}{V_P}\int_{P}f^2- \left(\frac{1}{V_P}\int_{P} f \right)^2.
\]
Note that the norm $\|(\cX,\cL)\|$ only agrees with the $L^2$ norm of $f$ when $f$ is normalized to have average zero.  Yao \cite[Theorem 5]{Yao17} gave a formula for the Ding-invariant for such a test configuration
\[
D^{\NA}(f) = -f(0) + \frac{1}{V_P}\int_{P}f.
\]
In what follows it will be useful to consider the modified non-Archimedean Ding functional.
\begin{dfn}
Let $g:P \rightarrow \mathbb{R}$ be any $L^2$ integrable function.  We define the $g$-modified non-Archimedean Ding functional by
\[
D^{\NA}_{g}(f) = -f(0) + \frac{1}{V_P}\int_{P}fg.
\]
\end{dfn}
\begin{rk} Note the $D^{\NA}_{g}(f)$ makes sense for any convex, $L^{2}$ integrable function $f:P\rightarrow \mathbb{R}$.
\end{rk}

Let $\eta \in \fs$ be the extremal vector field, and let $e$ be the associated linear function on $P$.  For convenience we set
\[
\ell := e+1.
\]
By definition, for any $T$-equivariant test configuration $(\cX,\cL)$ we have
\[
\begin{aligned}
D^{\NA}_{T}(\cX,\cL) &= D^{\NA}(\cX,\cL) - \langle (\cX,\cL) , \eta \rangle = D^{\NA}(f) + \frac{1}{V_P} \int_{P} f e\\
&= -f(0) + \frac{1}{V_P} \int_{X} f\ell = D^{\NA}_{\ell}(f)
\end{aligned}
\]

Yao \cite{Yao17} considered the optimization problem for the functions
\begin{equation*}
W(f):= \frac{D^\NA(f)}{\norm{f}} \ \ \ \text{and} \ \ \ W_\ell(f) := \frac{D^\NA_\ell(f)}{\norm{f}}, 
\end{equation*} 
defined on the space of $L^2$ integrable convex functions on $P$.  We first define
\begin{dfn}
The set of {\em balancing functions} on $P$ is the set
\[
\mathfrak{B} := \bigg\{ b \in L ^2(P) \ \bigg\vert \ b\geq0, \ \int_Pb=V_{P}, \ \text{and} \ \int_P bx_i=0 \ \text{for} \ 1\leq i \leq n \bigg\}.
\]
\end{dfn}

\begin{lem}[\cite{Yao17}]\label{lem: jensen}
For any $b \in \mathfrak{B}$ we have the Jensen-type inequality
\[
f(0) \leq \frac{1}{V_P}\int_P fb
\]
or any convex function $f$ on $P$.
\end{lem}

The main result of \cite{Yao17} is

\begin{thm}[\cite{Yao17}]\label{Yao}
There exists a piecewise linear convex function $d$ with the following properties: 
\begin{itemize} 
\item[$(1)$] 
The function $d$ is the unique minimizer of $W_\ell$ over the space of $L^2$ integrable convex functions satisfying the normalization 
$V_P D^\NA_\ell(d)=-\|d\|_{L^{2}(P)}^2$. 
The function $d$ satisfies $D^\NA_{d+\ell} (d+\ell)=0$, and 
\begin{equation*}
D^\NA_{d+\ell} (f) \geq 0 \qquad \text{ for all } f \in {\rm Conv}^{0}(P).
\end{equation*} 
In other words, $P$ is relatively stable for the $d$-modified non-Archimedean Ding functional. 
\item[$(2)$]
The function $d$ is orthogonal to the extremal affine function $\ell-1$ and $d +\ell-1$ is the unique minimizer of $W$ with normalization 
$V_P D^\NA(d+\ell-1)=-\|d+\ell-1\|_{L^{2}(P)}^2$. 
Moreover, $d+\ell$ is characterized as the $L^2$ minimal element in $\mathfrak{B}$. Then, $d+\ell$ is also characterized as the convex element achieving the equality $V_P b(0) = \int_P b^2$. 
\end{itemize} 
Furthermore,  $d+\ell$ is a simple convex function in the sense that $d+\ell = \max\{a, 0\}$ for some affine function $a$ on $P$. 
\end{thm} 

We now describe how to reduce the $MA^{-1}$-flow to an equation on the moment polytope; this material is standard and can be found, for example in \cite{Sze14}.  Let $X_0 = (\mathbb{C}^{*})^n$ be a dense, free open orbit of the torus $T$ inside of $X$.  Let $w_1,\ldots,w_n$ be coordinates on $(\mathbb{C}^{*})^n$.  On the covering space $\mathbb{C}^{n}$ we have coordinates $z_i = \xi_i +\sqrt{-1}\eta_i$ so that $e^{z_i} = w_i$.  An $(S^1)^n$ invariant metric $\omega$ on $X$ can be written on $X_0$ as $\omega = \dd \phi$ where $\phi$ depends only on $(\xi_1,\ldots,\xi_n)$, and $\phi: {\mathbb{R}^n} \rightarrow \mathbb{R}$ is strictly convex.  The moment map for the $(S^{1})^n$ action is 
\[
m(z_1,\ldots,z_n) = \left(\frac{\del \phi}{\del \xi_i}\right)_{1 \leq i \leq n }
\]
and the image of $m$ is the moment polytope $P\subset \mathbb{R}^{n}$.  The {\em symplectic potential} of $\omega$ is the Legendre transform of $\phi$.  Namely, for $x \in P$ there is a unique $\underline{\xi} \in \mathbb{R}^{n}$ such that $x = \nabla \phi(\underline{\xi})$ and we define
\[
u(x) = \sum_{i }x_i \underline{\xi_i} - \phi(\xi) = \sup_{\xi \in \mathbb{R}^{n}} \sum_{i }x_i \xi_i - \phi(\xi).
\]
The function $u: P\rightarrow \mathbb{R}$ is strictly convex function satisfies the Guillemin boundary conditions (see \cite{Sze14, Don02, Gui94}).  By general principles,  the Legendre transform of $u$ is $\phi$.  If we subtract from $u$ an affine function $u \mapsto u- \sum a_i x_i -b$ then by definition of the Legendre transform, $\phi \mapsto \phi(\xi_i+a_i)+b$.  Note that the change of variables $\xi_i \mapsto \xi_i +a_i$ corresponds to the action by $\tau = (e^{a_1}, \cdots, e^{a_n}) \in T$ so that $\psi(\xi_i+a_i)$ is the potential of the pulled-back metric $\omega_\tau=\tau^*\omega$. 
In particular we can subtract a linear function from $u$ so that it has the point $0 \in P$ as a minimizer, and this just corresponds to moving $\omega$ by an element of $T$.
\begin{dfn}
We say that a convex function $f: P\rightarrow \mathbb{R}$ is {\em normalized} if $f \geq 0$, and $f(0)=0$.
\end{dfn}

Any convex function $f$ on $P$ can be normalized by subtracting a supporting hyperplane at the origin.  In particular, if $f$ is $C^1$ and normalized then $\nabla f(0)=0$.  Furthermore, any symplectic potential $u$ can be normalized by using the $T$-action, which can be viewed as choice of good gauge.  In what follows we will denote by
\[
{\rm Conv}^{0}(P) := \left\{ u \in C^{0}(\overline{P}) : u \text{ is convex} \right\}.
\]
\cite[Proposition 3.3]{BB13} shows that every function $u \in {\rm Conv}^{0}(P)$ defines a positive current in $2 \pi c_1(X)$ with bounded potentials, however we will not need to use this fact. 

By the definition of the Legendre transform it is straightforward to compute that the Aubin-Yau functional is written on the polytope, in terms of the symplectic potential, as
\[
\delta E(u) = -\frac{1}{V_P}\int_{P} \delta u 
\]
and hence we have
\[
E(u) = -\frac{1}{V_P}\int_{P} u +C
\]
for some constant $C$.  Yao \cite{Yao17} showed that, up to an overall constant, the Ding functional is given by
\[
D(u) =  -\log \left(\int_{\mathbb{R}^{n}} e^{\phi- \inf_{\mathbb{R}^n} \phi} \right) - u(0) +\frac{1}{V_P} \int_{P} u.
\]
where $\phi$ is the Legendre transform of $u$.  We will also need the functional
\[
J_{T}(\omega) = \inf_{\tau \in T} \left( \sup_{X} \phi_{\tau} - E(\phi_{\tau})\right)
\]
where $\tau^{*}\omega_{\phi} = \omega_0 + \dd \phi_{\tau}$.  We have
\begin{lem}[Lemma 2.2 \cite{ZZ08}]
There is a uniform constant $C$ so that for any $(S^1)^n$ invariant K\"ahler metric $\omega$ we have
\[
J_{T}(\omega) \leq \frac{1}{V_P} \int_{P} \tilde{u} +C
\]
where $\tilde{u}$ is the normalized symplectic potential of $\omega$.  
\end{lem}

Following our discussion of the $g$-modified non-Archimedean Ding functional, it is natural to introduce the $g$-modified Ding functional
\[
\begin{aligned}
D_{g}(\omega) &= L(\phi) + \frac{1}{V_P} \int_{P}ug,\\
&=   -\log \left(\int_{\mathbb{R}^{n}} e^{-(\phi- \inf_{\mathbb{R}^n} \phi)} \right) - u(0) +\frac{1}{V_P} \int_{P} ug \\
&= -\log \left(\int_{\mathbb{R}^{n}} e^{-(\phi- \inf_{\mathbb{R}^n} \phi)} \right) +D_{g}^{\NA}(u)
\end{aligned}
\]
where, as before, $\phi$ is the Legendre transform of $u$.  Again, $D_{g}(u)$ is well-defined on ${\rm Conv}^{0}(P)$.  We will prove a coercivity result for the modified Ding energy, building on \cite{Nak17}. First, we prove a lemma which will play an important role in this section.

\begin{lem}\label{lem: sigBal}
Let $\omega_0$ be any $(S^{1})^n$ invariant metric, with symplectic potential $u_0$.  We define $\sigma: P\rightarrow \mathbb{R}$ by $\sigma_0(x) = e^{\rho_{0}(\xi)}$, where $\rho_0$ is the normalized Ricci potential of $\omega_0$, and $\xi = \nabla u_{0}(x)$; that is $x \longleftrightarrow \xi$ by the Legendre transform of $u_0$.  Then
\begin{enumerate}
\item $\sigma_0$ defines a positive balancing function.
\item The minimum of the modified Ding energy $D_{\sigma_{0}}(\cdot)$ over ${\rm Conv}^{0}(P)$ is attained by $u_0$.
\end{enumerate}
\end{lem}
\begin{proof}
For the first part, observe that
\[
\frac{1}{V_P}\int_{P}\sigma_0 =\frac{1}{V} \int_{X}e^{\rho_0} \omega_0^{n} = 1.
\]
Next we observe that, for the vector field $\mu_i \in \fs$ corresponding to the linear function $x_i: P \rightarrow \mathbb{R}$, we have
\[
\begin{aligned}
\langle \eta, \mu_i \rangle &= \frac{1}{V}\int_{X} h_{\mu_i} P_{\omega_0} (e^{\rho_0}-1) \omega_0^{n}\\
& = \frac{1}{V}\int_{X}h_{\mu_i} (e^{\rho_0}-1) \omega_0^{n}\\
&  = \frac{1}{V_P}\int_{P}x_i (\sigma_0 -1).
\end{aligned}
\]
On the other hand, from the definition of the extremal vector field $\eta$, and  we have
\[
\langle \eta, \mu_i \rangle = D^{\NA}(\mu_i) = -\frac{1}{V_P} \int_{P} x_i
\]
and so it follows that $\int_{P}\sigma_0 x_i =0$.  Since $\sigma_0 >\epsilon> 0$ is clear, we have established $(1)$.  To prove $(2)$, we first observe that from the definition of $D_{\sigma_0}$ we have
\[
\delta D_{\sigma_0}(\omega_0)=0.
\]
Next we show that $D_{g}(u(t))$ is convex along convex combinations in ${\rm Conv}^{0}(P)$.  Given any convex function $u' \in {\rm Conv}^{0}(P)$ we consider $u_t := (1-t)u_0+tu'$.  It is automatic that $\int_{P} g u_t$ is affine in $t$. Taking the Legendre transform of $u_t$ for each fixed $t$ we get a function $\phi(x, t)$ which is convex on $\mathbb{R}^{n}\times [0,1]$.  Since $u_t(0) = -\inf_{\mathbb{R}^n} \phi(x,t)$ we have that
\[
D_{g}(u_t) = -\log \left( \int_{x\in \mathbb{R}^{n}} e^{-\phi(x,t)} \right) +\frac{1}{V_P}\int_{P} g u_t 
\]
and so the convexity follows from Prekopa's theorem.  Specializing to $g= \sigma_0$, we get that $D_{\sigma_0}(u(t))$ is convex along the curve $u(t)$.  This establishes $(2)$.  We remark that this last argument is just a special case of Berndtsson's fundamental result \cite{Bern11}.
\end{proof}

Next we show that for any balancing function $b\in \cB$ the $b$-modified Ding functional is coercive.  While this result does not play a role in the proof of Theorem~\ref{thm: toric thm}, it is of independent interest, particularly because of the explicit dependence of the constant $\epsilon$ on the choice of $b$ (cf. \cite{DR15}).  The proof is a small modification of an argument of Nakamura \cite{Nak17}.

\begin{prop}
For any positive continuous balancing function $b \in \cB$ with $\inf_{P}b >0$ there are positive constants $\epsilon, C$ such that
\[
D_{b}(\omega) \geq \epsilon J_{T}(\omega)-C.
\]
Furthermore, we can take $\epsilon$ as close to $\inf_{P} b$ as we like.
\end{prop}
\begin{proof}
Fix $\epsilon_0 = \inf_{P} b >0$.  Let $u$ be any square integrable convex function on P.  If $g$ is any affine function on $P$, then it follows from the definition of a balancing function that
\[
D^{\NA}_{b}(u+g) = D^{\NA}_{b}(u).
\]
By subtracting a supporting tangent plane from $u$ we may assume that $u \geq 0$ and $u(0)=0$; in other words, $u$ is normalized.  For normalized $u$ we clearly have
\[
D^{\NA}_{b}(u) \geq \epsilon_0 \int_{P} u.
\]
Let $u$ be any smooth normalized symplectic potential, and let $\phi$ be the Legendre transform of $u$.  Then there is a constant $C$ so that we have
\[
\begin{aligned}
|D_{b}^{\NA}(u) - D^{\NA}_{\sigma_0}(u)| \leq C \int_{P} u &= C(k+1) \int_{P}u - Ck \int_{P}u\\
&\leq C(k+1) \epsilon_{0}^{-1} D_{b}^{\NA}(u) - Ck\int_{P} u.
\end{aligned}
\]
Letting $K = C(k+1)\epsilon_0^{-1}+1$ we have
\[
\begin{aligned}
D_{b}(u) &\geq -\log \left(\int_{\mathbb{R}^{n}}e^{-(\phi-\inf \phi)}\right)+ \frac{1}{K}D^{\NA}_{\sigma_0}(u) + \frac{Ck}{K} \int_{P}u\\
&\geq -\log \left(\int_{\mathbb{R}^{n}}e^{-(\phi_{K}-\inf \phi_{K})}\right)+ D^{\NA}_{\sigma_0}(\frac{u}{K}) + \frac{Ck}{K} \int_{P}u -n\log K\\
&= D_{\sigma_0}(\frac{u}{K})+ \frac{Ck}{K} \int_{P}u -n\log K
\end{aligned}
\]
where 
\[
\phi_{K}(\xi) = \frac{1}{K}\phi(K \xi)
\]
is the Legendre transform of $u/K$. In passing from the first line to the second line we used that
\begin{equation}\label{eq: covLeg}
\begin{aligned}
 -\log \left(\int_{\mathbb{R}^{n}}e^{-(\phi-\inf \phi)}\right) &\geq -\log \left(\int_{\mathbb{R}^{n}} e^{-\frac{1}{K}(\phi-\inf \phi)}\right)\\
 &= -\log \left(\int_{\mathbb{R}^{n}} e^{-(\phi_{K}-\inf \phi_{K})}\right) - n\log K.
 \end{aligned}
 \end{equation}
Thanks to the fact that $u/K \in {\rm Conv}^{0}(P)$, Lemma~\ref{lem: sigBal} implies that
\[
D_{\sigma_0}(\frac{u}{K}) \geq D_{\sigma_0}(u_0),
\] 
and so 
\[
D_{b}(u) \geq \frac{Ck}{K} \int_P u - n\log K -C.
\]

For any $\epsilon < \epsilon_0$ we can choose $k$ sufficiently large so that
\[
D_{b}(u) \geq \epsilon \int_{P}u -C.
\]
Finally, we note that from the definition of a balancing function $D_{b}(u)$ is unchanged by adding a linear function, as is $J_{T}$.  The proposition follows.
\end{proof}
We note two corollaries of the proof which will be used in the proof of Theorem~\ref{thm: toric thm}.

\begin{cor}\label{cor: toricDNA}
For any positive continuous balancing function $b \in \cB$ with $\inf_{P}b >0$, there are positive constants $\epsilon, C$ such that
\[
D_{b}(u) \geq \frac{\epsilon}{\sup_{P}b}D_{b}^{\NA}(u)-C.
\]
Furthermore, we can take $\epsilon$ as close to $\inf_{P} b$ as we like.
\end{cor}
\begin{proof}
First note that if $u$ is normalized then we have
\[
\epsilon \int_{P}u \geq \frac{\epsilon}{\sup_{P}b} \int_{P}bu =\frac{\epsilon}{\sup_{P}b} D^{\NA}_{b}(u),
\]
which proves the corollary when $u$ is normalized.  Since both sides of the inequality are unchanged by adding an affine function, the corollary follows.
\end{proof}

\begin{cor}\label{cor: toricEasy}
For any $K\geq 1$ there exists a constant $C_K$ such that
\[
-\log \left(\int_{\mathbb{R}^n} e^{-(\phi -\inf \phi)} \right) \geq -\frac{1}{K} D^{\NA}_{\sigma_0}(u) - C_K.
\]
\end{cor}
\begin{proof}
By Lemma~\ref{lem: sigBal}, the modified Ding functional $D_{\sigma_0}$ is bounded from below on ${\rm Conv}^{0}(P)$, and hence, for any $u\in {\rm Conv}^{0}(P)$ we have 
\[
\begin{aligned}
-\log \left(\int_{\mathbb{R}^n} e^{-(\phi_{K} -\inf \phi_{K})} \right)& =D_{\sigma_0}(\frac{u}{K}) -D^{\NA}_{\sigma_0}(\frac{u}{K})\\
& \geq -\frac{1}{K}D^{\NA}_{\sigma_0}(u) -  C.
\end{aligned}
\]
Where $\phi_{K}(\xi) = \frac{1}{K} \phi(K\xi)$ is the Legendre transform of $u/K$, with $\phi$ the Legendre transform of $u$.  We now use estimate~\eqref{eq: covLeg} to conclude.
\end{proof}

We now use our results to study the $MA^{-1}$-flow on toric manifolds.  A standard computation shows that, in terms of the symplectic potential, the flow is
\[
\ddt u = \sigma(t) -1
\]
where $\sigma(x,t) = e^{\rho(\xi,t)}$ and we use the gradient map of $u(t)$ to identify $x$ and $\xi$.  Along the $MA^{-1}$-flow the $g$-modified Ding functional satisfies
\begin{equation}\label{eq: modDingflow}
\begin{aligned}
\ddt D_{g}(u) &= -\frac{1}{V_{P}}\int_{P}(\sigma-1)^2 + \frac{1}{V_P}\int_{P}(\sigma-1)(g-1)\\
&= -\frac{1}{V_{P}}\int_{P}(\sigma-1)^2 + \frac{1}{V_P}\int_{P}(\sigma-1)g\\
&= \frac{1}{V_P} \int_{P}(\sigma-1)(g-\sigma)
\end{aligned}
\end{equation}
where we used that $\int_{P}\sigma = V_{P}$.  Note that
\[
R(t) = \frac{1}{V_P}\int_{P}(\sigma-1)^2 
\]
is just the Ricci-Calabi energy of the K\"ahler metric associated to $u(t)$.  From the second line of~\eqref{eq: modDingflow} and H\"olders inequality we have
\[
\ddt D_{g}(u) \leq -R(t) + \frac{1}{\sqrt{V_P}}\|g\|_{L^{2}(P)}R(t)^{\frac{1}{2}}.
\]
Thanks to the fact that $R(t)$ is decreasing along the flow by Corollary~\ref{cor: DconvFlow} there is uniform constant $C$, depending only on the initial data and $g$ so that
\begin{equation}\label{eq: toriclinub}
D_{g}(u) \leq C(1+t).
\end{equation}
Write
\[
(\sigma -1)(g-\sigma) = -(\sigma-g)^2 + g^2 -g\sigma -g +\sigma.
\]
By Theorem~\ref{Yao} the maximal destabilizer $d+\ell$ is characterized as the convex balancing function satisfying
\[
\frac{1}{V_{P}}\int_{P}(d+\ell)^2 =(d+\ell)(0) 
\]
and so substituting $g=d+\ell$ into~\eqref{eq: modDingflow}
\[
\ddt D_{d+\ell}(u) = -\frac{1}{V_P} \int_{P}(\sigma-(d+\ell))^2 -\left( (d+\ell)(0) - \frac{1}{V_{P}}\int_{P}(d+\ell)\sigma\right)
\]
where we used that $\int_{P} (d+\ell) = V_P = \int_{P}\sigma$.  On the other hand, by Lemma~\ref{lem: sigBal}, $\sigma$ is a balancing function, and so by Lemma~\ref{lem: jensen} it satisfies Jensen's inequality
\[
f(0) \leq \frac{1}{V_P} \int_{P} f\sigma
\]
for any convex $f$.  Since $d+\ell$ is convex, we conclude that
\begin{equation}\label{eq: keyToric}
\ddt D_{d+\ell}(u) \leq -\frac{1}{V_P} \int_{P}(\sigma-(d+\ell))^2. 
\end{equation}
We can now prove Theorem~\ref{thm: toric thm}.
\begin{proof}[Proof of Theorem~\ref{thm: toric thm}]
We begin by showing that $\liminf_{t\rightarrow \infty} \|\sigma(t)- (d+\ell)\|_{L^{2}(P)} =0$.  Suppose not.  Then by~\eqref{eq: keyToric} there are $\delta, T >0$, so that for all $t>T$ we have
\[
\ddt D_{d+\ell}(u) \leq -\delta
\]
and hence $D_{d+\ell}(u) \leq -\delta t + C'$ for $t \geq T$.  On the other hand by Corollary~\ref{cor: toricEasy} we have
\[
\begin{aligned}
D_{d+\ell}(u) &= -\log \left( \int_{x\in \mathbb{R}^{n}} e^{-(\phi(x,t) - \inf_{x\in \mathbb{R}^{n}} \phi(x,t))}\right) +D^{\NA}_{d+\ell}(u)\\
&\geq -\frac{1}{K} D^{\NA}_{\sigma_0}(u)- C_{K} + D^{\NA}_{d+\ell}(u),
\end{aligned}
\]
for any $K\geq 1$.  Since $d+\ell$ is a balancing function we have $D^{\NA}_{d+\ell}(u) \geq 0$ and so thanks to ~\eqref{eq: toriclinub} and Corollary~\ref{cor: toricDNA} we have
\[
D_{d+\ell}(u) \geq -\frac{1}{K}(C(1+t)) - C_{K}.
\]
Taking $K$ sufficiently large so that $C/K \ll \delta$, we obtain contradiction.  Thus, there is a sequence of times $t_{k} \rightarrow +\infty$ so that $\sigma(t_k) -1 \longrightarrow d+e$ in $L^{2}(P)$.  In particular, we have
\[
\lim_{k\rightarrow \infty} \|\sigma(t_k) -1\|_{L^{2}(P)} = \|d+e\|_{L^{2}(P)}.
\]
Since $\|\sigma(t)-1\|_{L^{2}(P)}$ is decreasing along the flow (as the Ricci-Calabi energy decreases), we conclude that $\|\sigma(t)-1\|_{L^{2}(P)} \rightarrow  \|d+e\|_{L^{2}(P)}$.  Using that $\int_{P} \sigma = V_P = \int_{P} (d+\ell)$ we get
\[
\|\sigma(t)\|_{L^{2}(P)} \rightarrow \|d+\ell\|_{L^{2}(P)}.
\]

We now extend the convergence to the whole sequence.  By the parallelogram identity we have
\[
\|\sigma -(d+\ell)\|_{L^{2}(P)}^2 = 2 \left( \|\sigma\|_{L^{2}(P)}^{2} + \|d+\ell\|_{L^{2}(P)}^2\right) - \|\sigma +(d+\ell)\|^{2}_{L^{2}(P)}.
\]
By Theorem~\ref{Yao}, $d+\ell$ is the balancing function with minimal $L^2$ norm.  But by Lemma~\ref{lem: sigBal} $\sigma$ is also a balancing function.  Since the set of balancing functions is clearly convex, we have
\[
\| \frac{\sigma +(d+\ell)}{2}\|_{L^{2}(P)}^2 \geq \|d+\ell\|^{2}_{L^{2}(P)}
\]
and so it follows that
\[
\|\sigma -(d+\ell)\|_{L^{2}(P)}^2 \leq  2 \left( \|\sigma\|^{2} - \|d+\ell\|_{L^{2}(P)}^2\right). 
\]
But we have already shown that the right hand side converges to zero.  Thus $\sigma-1$ converges to $d+e$ in $L^{2}(P)$.  We get
\[
\begin{aligned}
\lim_{t\rightarrow \infty} R(t)^{1/2} &= \lim_{t\rightarrow \infty} \bigg[ \frac{1}{V_P} \int_{P} (\sigma -1)^{2} \bigg]^{1/2} \\
& = \frac{1}{V_P^{1/2}} \|d+e\|_{L^{2}(P)} = \sup_{f} \frac{-D^{\NA}(f)}{V_P^{-1/2}\|f\|}
\end{aligned}
\]
where the sup is over all $L^2$ integrable convex functions on $P$.
\end{proof}

We finish by making some remarks on the picture suggested by Theorem~\ref{thm: toric thm}.  Suppose $X$ is an unstable toric Fano manifold, and let $d+\ell \ne 0$ be the maximal destabilizer.  By Theorem~\ref{Yao} we have
\[
\int_{P}(d+\ell)^{2} = V_P(d+\ell)(0), \qquad d+\ell = \max \{a, 0\}
\]
where $a$ is some affine function.  The set $a \geq 0$ defines a subpolytope $P'\subset P$ with $0\in P'$.  Furthermore, $d+\ell|_{P'} = a|_{P'}$ is a linear function.  If $a$ is rational, then $P'$ is a toric log Fano variety \cite[Proposition 3.2]{BB13}.  By \cite{Yao17}, the polytope $P'$ is relatively stable with respect to the affine function $d+\ell$, and hence we expect it to admit a (singular) Mabuchi soliton with extremal vector field $d+\ell$.  If $P\cap (P')^c$ has positive measure then some mass, equal to the Lebesgue measure of the complement of $P'$ in $P$, must be lost in the limit, and we expect this to be the result of the weak limit of the flow developing non-trivial Lelong numbers.  Note that in the unstable case the test configuration corresponding to the destabilizer $d+\ell = \max\{a,0\}$ is non-trivial and has non-normal central fiber.  On the other hand, the degeneration produced by Chen-Sun-Wang \cite{CSW15} using the K\"ahler-Ricci flow is always a product test configuration.

Finally, we remark that by Theorem~\ref{thm: toric thm}, in the unstable case we see that the Ricci potential must diverge to $-\infty$.  In particular, the analogue of Perelman's estimates for the K\"ahler-Ricci flow do not hold for the $MA^{-1}$-flow.

\end{document}